\documentclass[onefignum,onetabnum]{siamart171218}

\usepackage{amsmath,amsfonts,amssymb,amsopn}
\usepackage{enumerate}
\usepackage{graphicx}
\usepackage{xcolor}
\usepackage{tikz-cd}
\usepackage{tabularx}
\usepackage{arydshln}

\usepackage{setspace}
\usepackage{marginnote}

\def\fdeg{r}

\def\llangle{\langle \! \langle}
\def\rrangle{\rangle \! \rangle}

\def\S2{{S^2}}
\def\Psiec{\Psi \mathrm{ec}(S^2)}
\def\dd{\mathrm{d}}
\newcommand{\M}{\mathcal{M}}

\newcommand{\R}{\mathbb{R}}

\DeclareMathOperator*{\spann}{span}

\newsiamremark{example}{Example}

\newsiamremark{remark}{Remark}
\newsiamremark{hypothesis}{Hypothesis}
\crefname{hypothesis}{Hypothesis}{Hypotheses}
\newsiamthm{claim}{Claim}

\headers{A Local Spectral Exterior Calculus for the Sphere}{C. Carvalho da Silva, C. Lessig, B. Dodov, H. Dijkstra, T. Sapsis}

\title{A Local Spectral Exterior Calculus for the Sphere and Application to the Rotating Shallow Water Equations}

\author{Clauson Carvalho da Silva\thanks{Otto-von-Guericke-Universit{\"a}t Magdeburg, Magdeburg, Germany (C. Lessig: \email{lessig@isg.cs.uni-magdeburg.de}, \url{http://graphics.cs.uni-magdeburg.de/}).}
\and Christian Lessig\footnotemark[1]
\and Boyko Dodov\thanks{AIR Worldwide}
\and Hendrik Dijkstra\thanks{Institute for Marine and Atmospheric Research, Utrecht University}
\and Themis Sapsis\thanks{Sandlab, Massachusetts Institute of Technology}}

\usepackage{amsopn}

\ifpdf
\hypersetup{
  pdftitle={A Local Spectral Exterior Calculus for the Sphere and Application to the Shallow Water Equations},
  pdfauthor={C. Lessig, C. Carvalho da Silve, B. Dodov, H. Dijkstra, T. Sapsis}
}
\fi

\begin{document}

\maketitle

\begin{abstract}
  We introduce $\Psiec$, a local spectral exterior calculus for the two-sphere $\S2$.
  $\Psiec$ provides a discretization of Cartan's exterior calculus on $\S2$ formed by spherical differential $r$-form wavelets $\smash{\psi_{jk}^{r,\nu}}$.
  These are well localized in space and frequency and provide (Stevenson) frames for the homogeneous Sobolev spaces $\smash{\dot{H}^{-r+1}( \Omega_{\nu}^{r} , \S2 )}$ of differential $r$-forms.
  At the same time, they satisfy important properties of the exterior calculus, such as the de Rahm complex and the Hodge-Helmholtz decomposition.
  Through this, $\Psiec$ is tailored towards structure preserving discretizations that can adapt to solutions with varying regularity.
  The construction of $\Psiec$ is based on a novel spherical wavelet frame for $L_2(S^2)$ that we obtain by introducing scalable reproducing kernel frames.
  These extend scalable frames to weighted sampling expansions and provide an alternative to quadrature rules for the discretization of needlet-like scale-discrete wavelets.
  We verify the practicality of $\Psiec$ for numerical computations using the rotating shallow water equations.
  Our numerical results demonstrate that a $\Psiec$-based discretization of the equations attains accuracy comparable to those of spectral methods while using a representation that is well localized in space and frequency.
\end{abstract}



\section{Introduction}
\label{sec:intro}

Adaptivity and structure preservation are important objectives for the discretization of partial differential equations.
Adaptivity is a prerequisite for optimal convergence rates when the regularity of a solution varies, that is for the efficiency of a numerical scheme.
For it, a representation that, in an appropriate sense, can ``zoom in'' on irregular features is required.
Structure preservation, which means that a discretization preserves essential aspects of a continuum theory, for example its conservation laws, plays a critical role for qualitatively correct solutions.
It relies on a representation that respects Cartan's exterior calculus of differential forms, e.g. the de Rahm complex.
On the two-sphere $\S2$, one application where adaptivity and structure preservation are of great importance are weather and climate simulations.
Adaptivity ensures there that local phenomena, e.g. extreme events such as hurricanes, are efficiently resolved ~\cite{Behrens2006,Zarzycki2015,Kevlahan2019} while structure preservation is needed for the conservation of energy and other invariants during the very long integrations times frequently required~\cite{Cotter2012,Cotter2014}.

To obtain numerical schemes for $S^2$ that are both adaptive and structure preserving, we introduce $\Psiec$, a local spectral exterior calculus for the sphere.
Its central objects are spherical differential $\fdeg$-form wavelets $\smash{\psi_{jk}^{\fdeg,\nu}(\omega)}$ that span the spaces $\Omega^{\fdeg}$ of differential $r$-forms and satisfy important properties of Cartan's exterior calculus (see Table~\ref{tab:notation} for our notation).
The wavelets $\smash{\psi_{jk}^{\fdeg,\nu}(\omega)}$ are localized around spherical harmonics frequencies $\smash{l = 2^{j}}$ and locations $\smash{\lambda_{jk}} \in \S2$, which enables them to adapt to local irregularities in a signal.
Their construction simultaneously ensures that they satisfy the de Rahm complex, e.g. the idempotence of the exterior derivative, $\dd \! \cdot \! \dd = 0$, and the Hodge-Helmholtz decomposition.
In contrast to existing discretizations, such as Finite Element Exterior Calculus~\cite{Arnold2018} or the TRiSK scheme~\cite{Thuburn2009,Ringler2010}, the differential $\fdeg$-form wavelets $\smash{\psi_{jk}^{\fdeg,\nu}(\omega)}$ of $\Psiec$ are bona fine forms in the sense of the continuous theory.
Hence, all operations from there are well defined and we show that most are also closed in $\Psiec$, and can therefore also be computed efficiently, cf. Fig.~\ref{fig:psiec_s2}.

The construction of $\Psiec$ is based on a new, scalar discrete wavelet frame for $L_2(S^2)$.
It is a discretization of scale-discrete, harmonic wavelets defined by window coefficients $\kappa_l^j$ bandlimited in the spherical harmonics domain.
Such wavelets have been proposed, for example, by McEwen, Durastani, and Wiaux~\cite{McEwen2016} and under the name needlets by Narcowich, Petrushev, and Ward~\cite{Narcowich2006a}.
Instead of using quadrature rules as for needlets~\cite{Narcowich2006a}, we discretize the scale-discrete wavelets, however, using scalable reproducing kernel frames, a concept we introduce for this purpose.
These are formed by quasi-uniform locations $\lambda_{jk} \in S^2$ and positive scaling factors $w_{jk} \in \R^+$ and extend the scalable frames recently introduced by Kutyniok, Okoudjou and co-workers~\cite{Kutyniok2013,Chen2015} to weighted sampling expansions.
Similar to spherical $t$-designs, we currently do not have theoretical guarantees for the existence of  scalable reproducing kernel frames $\{ (  \lambda_{jk} , w_{jk} ) \}_{k \in \mathcal{K}_j}$ for the spaces spanned by the wavelets.
However, we present a numerical algorithm that allows one to obtain them up to large degree $L$.

\begin{figure}[t]
  \includegraphics[width=\textwidth]{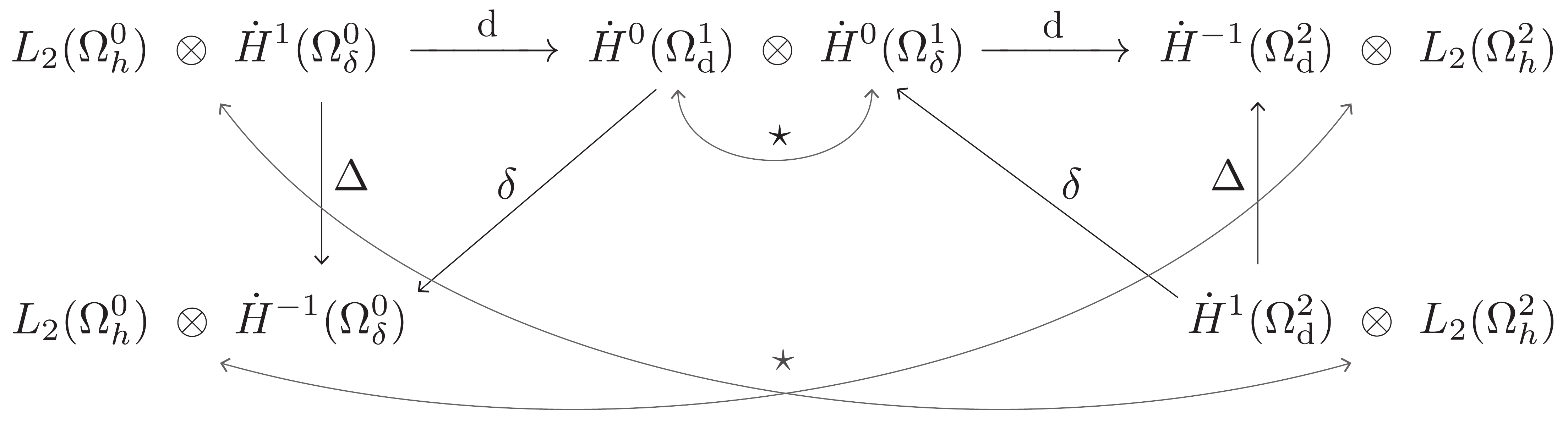}
  \caption{Conceptual description of our local spectral exterior calculus $\Psiec$ and its closures, indicated by arrows, which can be read similar to a commutative diagram.
  $\Psiec$ provides a discretization of the de Rahm complex, i.e. the co-chain complex of differential forms $\alpha \in \Omega^\fdeg(\S2)$ under the exterior derivative $\dd : \Omega^\fdeg(\S2) \to \Omega^{\fdeg+1}(\S2)$.
  In $\Psiec$, $\dd$ hence maps discrete $\fdeg$-forms to discrete $(\fdeg+1)$-forms and also $\dd \cdot \dd = 0$ holds.
  Our construction also respects the Hodge-Helmholtz decomposition and we have distinct representations for exact, co-exact and harmonic forms, that is for the spaces $\Omega_{\dd}^{\fdeg}$, $\Omega_{\delta}^{\fdeg}$, and $\Omega_{h}^{\fdeg}$.
  This provides us, for example, with control over the domain, image and kernel of the exterior derivative.
  $\Psiec$ is also closed under metric-dependent operators such as the co-differential $\delta : \Omega^{\fdeg+1}(\S2) \to \Omega^{\fdeg}(\S2)$ and the Laplace-Beltrami operator $\Delta : \Omega^{\fdeg}(\S2) \to \Omega^{\fdeg}(\S2)$ for $0$- and $2$-forms.
  The functional analytic setting of $\Psiec$ are the homogeneous Sobolev spaces $\dot{H}^{s}(\Omega_{\nu}^{\fdeg},\S2)$ and our discrete wavelet differential forms $\smash{\psi_{jk}^{\fdeg,\nu}(\omega)}$ provide (Stevenson) frames for these.
  The formal characterization of $\Psiec$ can be found in Theorem~\ref{thm:diff_form_wavelets:frame_property} and Theorem~\ref{thm:psiec}.
  }
  \label{fig:psiec_s2}
\end{figure}

The second building for the construction of $\Psiec$ is a spectral exterior calculus for $S^2$ that we introduce.
For scalar $0$- and $2$-forms, its basis functions are the usual spherical harmonics, i.e. $\smash{y_{lm}^{0,\delta} \equiv y_{lm}(\omega)}$ and $\smash{y_{lm}^{2,\dd} \equiv y_{lm}(\omega) \, d\omega}$, and for exact and co-exact $1$-forms the bases are obtained through the exterior derivative $\dd$ as $\smash{y_{lm}^{1,\dd}(\omega) \equiv \dd y_{lm}^{0,\delta}(\omega)}$ and $\smash{y_{lm}^{1,\delta}(\omega) = \star \dd y_{lm}^{0,\delta}(\omega)}$, which are covariant analogues of classical vector spherical harmonics.
In Theorem~\ref{prop:spectral_ec:properties} we show that the spectral differential forms $\smash{y_{lm}^{r,\nu}(\omega)}$ provide a near perfect discretization of the exterior calculus with, for example, closure of the de Rahm complex, a diagonal Hodge dual, and distinct basis functions for exact, co-exact and harmonic forms, i.e. for the spaces in the Hodge-Helmholtz decomposition.

\begin{table}
\begin{tabularx}{\textwidth}{ |l|X| }
   \hline
   $\omega = (\theta , \phi)$ & spherical coordinates for $S^2$ with $\theta \in [0,\pi]$ and $\phi \in [0,2\pi)$
   \\[2pt]
   $y_{lm}(\omega)$ & (Legendre) spherical harmonics
   \\[2pt]
   $\mathcal{H}_l(S^2)$ & space spanned by all spherical harmonics in band $l$
   \\[2pt]
   $\psi_{jk}(\omega)$ & scalar spherical wavelets at level $j$ and location $\lambda_{jk} \in S^2$
   \\[2pt]
   $\kappa_l^j$ & window coefficients for spherical (differential $r$-form) wavelets
   \\[2pt]
   $\{ (  \lambda_{jk} , w_{jk} ) \}_{{k \in \mathcal{K}_j}}$ & scalable reproducing kernel frame for $\mathcal{H}_{\leq L_j}$
   \\[2pt]
   $\Lambda_{j} = \{ \lambda_{jk} \}$ & locations of scalable reproducing kernel frame on level $j$
   \\[2pt]
   $\mathfrak{X}(\S2)$ & space of all vector fields on $\S2$
   \\[2pt]
   $\mathfrak{X}_{\mathrm{div}}(\S2)$ & space of divergence free vector fields on $\S2$
   \\[2pt]
   $\nu \in \{ \dd , \delta , h \}$ & $\fdeg$-form type w.r.t. to Hodge-Helmholtz decomposition
   \\[2pt]
   $\bar{\nu} \in \{ \delta , \dd , h \}$ & $\fdeg$-form type under Hodge dual
   \\[2pt]
   $\Omega_{\nu}^\fdeg(\S2)$ & space of differential $r$-forms of type $\nu$
   \\[3pt]
   $y_{lm}^{r \nu}(\omega)$ &  spectral differential $r$-form basis functions
   \\[4pt]
    $\psi_{jk}^{r \nu}(\omega)$ &  spherical differential $r$-form wavelet at level $j$ and $\lambda_{jk} \in \S2$
   \\[3pt]
    $\dot{H}^{s}(\Omega_{\nu}^r,\S2)$ & homogeneous Sobolev space of order $s$ for $r$-forms
   \\[2pt]
   $\bar{\alpha}$ & basis function coefficient vector for differential form $\alpha$
   \\[2pt]
   \hline
\end{tabularx}
  \caption{Notation used throughout the paper.}
  \label{tab:notation}
\end{table}

With the scalar wavelets $\psi_{jk}(\omega)$ and the $\smash{y_{lm}^{r,\nu}(\omega)}$, the differential $\fdeg$-form wavelets $\smash{\psi_{jk}^{\fdeg,\nu}(\omega)}$ that form $\Psiec$ are constructed by using the bandlimited window coefficients $\kappa_l^j$ of the $\psi_{jk}(\omega)$ together with the spectral form basis functions $\smash{y_{lm}^{r,\nu}(\omega)}$, i.e. the mother wavelets are given by
\begin{align}
  \psi_{jk}^{\fdeg,\nu}(\omega) = \sum_{l=0}^{\infty} \sum_{m=-l}^l a_l \, \kappa_l^{j} \, y_{lm}^{\fdeg,\nu}(\omega)
\end{align}
where $a_l$ is a weighting factor.
The spatial discretization is again provided by scalable reproducing kernel frames $\{ (  \lambda_{jk} , w_{jk} ) \}_{k \in \mathcal{K}_j}$.
By linearity, the differential $\fdeg$-form wavelets $\smash{\psi_{jk}^{\fdeg,\nu}(\omega)}$ are closed in the de Rahm complex.
Since we have distinct form wavelets for exact, co-exact and harmonic forms, namely, $\smash{\psi_{jk}^{\fdeg,\dd}(\omega)}$, $\smash{\psi_{jk}^{\fdeg,\delta}(\omega)}$, and $\smash{\psi_{jk}^{\fdeg,h}(\omega)}$, they intrinsically also respect the Hodge-Helmholtz decomposition, that is, we have explicit control of the domain, image and kernel of the exterior derivative.
For exact and co-exact forms, i.e. $\nu = \{ \dd , \delta \}$, the wavelets $\smash{\psi_{jk}^{\fdeg,\nu}(\omega)}$ provide (Stevenson) frames for the homogeneous Sobolev spaces $\dot{H}^{-\fdeg+1}(\Omega_{\nu}^{r},\S2)$, which are a natural functional analytic setting when harmonic forms are treated separately.
The use of Stevenson frames~\cite{Stevenson2003,Balazs2019}, i.e. leaving dual frame functions $\smash{\tilde{\psi}_{jk}^{\fdeg,\nu}(\omega)}$ in the dual space $\dot{H}^{r-1}(\Omega_{\nu}^{r},\S2)$, provides thereby the advantage that we obtain closure under the Hodge dual and for the Laplacian for $0$- and $2$-forms.
An overview of the relationship satisfied by $\Psiec$ is provided in Fig.~\ref{fig:psiec_s2}; the precise statements can be found in Theorem~\ref{thm:diff_form_wavelets:frame_property} and Theorem~\ref{thm:psiec} in Sec.~\ref{sec:forms}.

To demonstrate the practicality of $\Psiec$, we use it for the discretization of the rotating shallow water equations on the sphere $\S2$.
Numerical experiments demonstrate that our $\Psiec$-based discretization attains accuracy that is comparable to spectral methods for standard test cases~\cite{Williamson1992} as well as for forecast experiments.
With very good energy and enstrophy conservation, our experiments also demonstrate $\Psiec$'s potential for structure preserving numerical schemes.
We leave an investigation of adaptivity to future work.

The remainder of the paper is structured as follows.
In Sec.~\ref{sec:related} we discuss related work.
Subsequently, in Sec.~\ref{sec:wavelets} we construct the scalar discrete wavelet frame for $L_2(S^2)$.
The spectral exterior calculus for $S^2$ is introduced in Sec~\ref{sec:forms:spectral} followed by the local spectral exterior calculus $\Psiec$ in the remainder of the section.
The discretization of the rotating shallow water equation using $\Psiec$ is presented in Sec.~\ref{sec:shallow}.
We summarize the notation used throughout the paper in Table~\ref{tab:notation}.

\section{Related Work}
\label{sec:related}

Our work builds on various directions in the literature. 
We will discuss those most pertinent to our construction of $\Psiec$ and the scalar wavelets it is build on.

\subsection{Wavelets for $S^2$}

Wavelets strive for a compromise between spatial and frequency localization to be able to adaptively ``zoom in'' on irregular features. 
There is a considerably body of work on wavelets on the sphere $\S2$, see e.g.~\cite{Starck2006} and~\cite{Antoine2007} for surveys.
Closely related are also multi-scale radial basis function schemes for $\S2$, in particular since there is no uniform grid on $\S2$ and any discrete wavelet representation uses, in a certain sense, scattered locations. 
We will hence also briefly discuss spherical radial basis functions in the following.
For the purposes of the present work a wavelet representation will be called discrete when both the set of levels $j$ and the set of locations $\lambda_{jk} \in \S2$ on each level are discrete.

\subsubsection{Scalar wavelets for $S^2$}

Scalar, discrete wavelets on $\S2$ fall, broadly speaking, into two categories. 
The first one are subdivision-based wavelets, e.g.~\cite{Schroeder1995,Bonneau1999}.
For these one constructs a hierarchical partition of $\S2$, e.g. using on a subdivision scheme for a platonic solid, and the wavelets are defined based on it. 
Higher-order wavelets can thereby be obtained using the lifting scheme~\cite{Sweldens1996}.
For the second category, sometimes referred to as harmonic wavelets~\cite{Potts1996}, the wavelet functions are constructed in spherical harmonics space, i.e. in the frequency domain on $\S2$. 
Through this, these wavelets typically have excellent frequency localization.
However, the construction of harmonic wavelets with compact support (i.e. spherical Daubechies-type wavelets) is at the moment an open problem.
In contrast, subdivision-based wavelets are inherently compactly supported but they suffer from limited frequency localization.

The scalar wavelets we will introduce in Sec.~\ref{sec:wavelets} fall into the second category of harmonic wavelets. 
Their constructions is inspired by the needlets by Narcowich, Petrushev, and Ward~\cite{Narcowich2006a}.
However, instead of using quadrature rules for the discretization as in~\cite{Narcowich2006a}, we use scalable reproducing kernel frames. 
The wavelets are characterized by bandlimited, discrete window functions $\smash{\kappa_l^j}$ in the spherical harmonics domain and with a suitable choice quasi-exponential localization in the spatial domain can be attained~\cite{Narcowich2006a,McEwen2016}.
We will use the $\smash{\kappa_l^j}$ by McEwen, Durastani and Wiaux~\cite{McEwen2016} for our numerical examples. 

An alternative to wavelets are spherical radial basis functions (RBF), e.g.~\cite{Hubbert2001,Golitschek2001,LeGia2010}, which can also be constructed in a multi-scale framework, e.g.~\cite{LeGia2012}.
For these, one typically uses compactly supported spatial windows.  
Reconstruction, however, requires the solution of a linear system, which makes the approach ill suited when the number of points becomes very large. 
With our wavelets, a linear solve is avoided by using judiciously chose locations $\lambda_{jk} \in \S2$ and introducing weights $w_{jk}$ that provide additional degrees of freedom.
The $\lambda_{jk}$ and $w_{jk}$ form together the aforementioned scalable reproducing kernel frames that are the key to obtain a Parseval tight wavelet frame, and hence allow for self-dual reconstruction without the need for solving a linear system.

\subsubsection{Vector-valued wavelets for $S^2$}

As in the Euclidean case, vector-valued wavelets on $\S2$ received only little attentions in the literature. 
A construction is sketched by Freeden and co-workers~\cite[Ch. 13.3]{Freeden1998} but, to our knowledge, these were never implemented. 
Recently, Fuselier et al.~\cite{Fuselier2009,Fuselier2009a} considered vector-valued RBF interpolation of vector fields and proved error estimates for Sobolev data.
Their construction is similar to ours and they also obtain curl- and divergence free basis functions.
However, since they are in an RBF framework where the kernels are at arbitrary locations, reconstruction or the construction of dual functions requires the solution of a linear system.
Li, Broadbridge, Olenko, and Wang~\cite{Li2019} recently introduced an extension of needlets to vector fields of $\S2$ and studied fast algorithms for the projection and reconstruction. 
For our wavelets, fast algorithms still need to be considered but our differential form wavelets are in the larger context of the de Rahm complex.

%

\subsection{Discretizations of Exterior Calculus}

Numerical formulations that satisfy important properties of Cartan's exterior calculus, such as the de Rahm complex or Stokes' theorem, have been developed in various fields and come under different names. 
Their common objective is to obtain discrete systems that closely mimimic continuous partial differential equation, e.g. have the same or analogous conservation laws, and through this lead to better numerical perfomance, see e.g.~\cite{Boffi1999} or~\cite[Ch. 1]{Arnold2018}.

To our knowledge, the first structure preserving discretizations that can be found in the literature are the spectral methods developed for weather and climate simulation, e.g. for the barotropic vorticity equation~\cite{Silberman1954,Platzman1960}.
Although it was observed early on that these conserve energy and other invariants~\cite{Baer1961}, the underlying reasons were not studied systematically.
The spectral exterior calculus that we develop in Sec.~\ref{sec:forms:spectral} provides, in retrospect, some insight into this behavior.
Conservation properties were considered explicitly in the development of the so-called Arakawa grids~\cite{Arakawa1977} that carefully associate physical quantities with either vertices, edges or faces of a mesh~\cite{Pletzer2019}. 
This can be interpreted as distinguishing differential forms of different degree in the discretization and is a hallmark of numerical exterior calculus.
The explicit connections to differential forms was, however, only made much later.

Motivated by applications in electromagnetics and elasticity, N{\'e}d{\'e}lec~\cite{Nedelec1980,Nedelec1986} developed mixed finite elements for $\R^2$ and $\R^3$ that respect the structure of the differential operators of vector calculus by, effectively, associating $1$-forms with edges and $2$-forms with faces.
He did not make the connection to exterior calculus, although it is well known that in $\R^2$ and $\R^3$ vector and exterior calculus are isomorphic.
For computational electromagnetics, Bossavit developed the ideas further~\cite{Bossavit1997} and, to our knowledge, was the first who recognized the connection to Whitney forms~\cite{Whitney1957}, and hence to continuous differential forms.
The approach was brought into a mathematically more rigorous formulation by Hiptmaier~\cite{Hiptmair2001}. 

An alternative discretization of Cartan's exterior calculus using Whitney forms is Discrete Exterior Calculus (DEC) by Hirani, Desbrun and co-workers~\cite{Hirani2003,Desbrun2006} 
Recently, Budninskiy, Owahdi, and Desbrun~\cite{Budninskiy2019} extended this to a wavelet-based discrete exterior calculus with subdivision-type wavelets defined based on a multi-resolution mesh.
This is similar to the WaveTRiSK scheme proposed by Dubos and Kevlahan~\cite{Dubos2013,Aechtner2015} where subdivision wavelets are defined on staggered Arakawa C-grids. 
The works by Budninskiy, Owahdi, and Desbrun~\cite{Budninskiy2019} and Dubos and Kevlahan~\cite{Dubos2013,Aechtner2015} are, to our knowledge, the only ones in the literature where adaptivity and structure preservation have been considered together.

Finite Element Exterior Calculus (FEEC) by Arnold and co-workers~\cite{Arnold2006,Arnold2018} is another discretization of differential forms and the associated calculus. 
This work provides a comprehensive treatment of the subject, including of the functional analytic setting not considered in most other works. 
A similar approach are mimetic discretizations by Bochev and Hyman~\cite{Bochev2006}.

To improve convergence rates, also higher order finite element-type discretization of exterior calculus have been considered, see e.g.~\cite{Rapetti2009,Hiptmair2001,Rufat2014,Gross2018a}. 
Our work also provides a higher order discretization and our numerical results indicate that we attain the same convergence rates as spectral methods.

A fundamental difference between the above discretizations and our work is that wavelet differential $\fdeg$ forms $\psi_{jk}^{\fdeg,\nu}(\omega)$ are $\fdeg$-forms in the sense of the continuous theory.
They hence also satisfy the de Rahm complex in this sense.
In finite element-type discretizations such as DEC and FEEC, in contrast, one constructs a discrete structure with the same algebraic properties as the de Rahm complex.
Closely related to our work is in this respect is the ``spectral exterior calculus'' recently proposed by Berry and Giannakis~\cite{Berry2018}.
However, these authors are concerned with applications to manifold learning and did not consider localization.
Lessig~\cite{Lessig2018psiec} recently proposed a wavelet-based discretization of exterior calculus for $\R^2$ and $\R^3$ that also uses continuous differential form basis functions.
This work relies heavily on the structure of the exterior calculus in the Fourier domain, which is not available on the sphere.
In~\cite{Lessig2018psiec} also no numerical results were presented. 

\subsection{Computational Models for the Shallow Water Equation}

The rotating shallow water equations are a simplified, 2D model for atmospheric dynamics~\cite{Zeitlin2018}. 
Their discretization often serves as a stepping stone for the development of more complex schemes, and a correspondingly large number of approaches have been proposed in the literature.
Spectral models for the shallow water equations became practical with the advent of the fast transform method~\cite{Orszag1970,Eliasen1970} in the early 1970s and were developed, e.g., by Bourke~\cite{Bourke1972}.
The first mesh-based method for the shallow water equation with conservation properties was those by Arakawa and Lamb based on the C-grid~\cite{Arakawa1981}.
Taylor, Tribbia and Iskandarani~\cite{Taylor1997} develop a spectral element model, that aims at combining the advantages of finite-element and higher-order methods with faster convergence.

Dubos and Kevlahan~\cite{Dubos2013,Aechtner2015} recently proposed a method that constructs wavelets based on a multi-resolution C-grid and uses the TRiSK scheme~\cite{Thuburn2009,Ringler2010} to obtain structure preservation. 
An overview over other, mesh-based multi-resolution schemes is provided by Behrens~\cite{Behrens2006}.
Related to our approach is also a scheme proposed by Schwarztrauber~\cite{Swarztrauber2004} that uses vector spherical harmonics, i.e. the contravariant analogues of the $1$-form basis functions $y_{lm}^{1,\dd}(\omega)$ and $y_{lm}^{1,\delta}(\omega)$. 


\section{A Parseval Tight Discrete Wavelet Frame for $L_2(S^2)$}
\label{sec:wavelets}

The construction of our isotropic wavelet frame for scalar functions proceeds as follows.
The mother wavelets $\psi_j(\omega)$ are centered at the North Pole and defined through window coefficients $\smash{\kappa_l^j}$ in the spherical harmonics domain.
To cover the whole sphere and allow for the representation of arbitrary signals, the isotropic $\psi_j(\omega)$ are rotated to judiciously chosen locations $\lambda_{jk} \in S^2$ that are part of a scalable reproducing kernel frame $\{ (\lambda_{jk} , w_{jk}) \}_{k \in \mathcal{K}_j}$.
The use of these is the key to the Parseval tightness of the discrete wavelet representation.

After introducing notation, we will in the following first construct the scalable reproducing kernel frames.
With them in hand, we will be able to obtain the wavelet frame in Lemma~\ref{thm:wavelets}.

\subsection{Notation}
\label{eq:wavelets:notation}

Let $\S2$ be the unit sphere.
We will work with spherical coordinates where $\theta \in [0,\pi]$ is the angle to the $x_3$-axis and $\phi \in [0,2\pi]$ the azimuthal one in the $x_1$-$x_2$ plane.

The analogue of the Fourier transform on the sphere is the spherical harmonics expansion.
For any $f \in L_2(S^2)$ it is given by
\begin{align}
  \label{eq:spherical_harmonics:expansion}
  f(\omega)
  &= \sum_{l=0}^{\infty} \sum_{m=-l}^l \underbrace{\langle f(\eta) , y_{lm}(\eta) \rangle}_{\displaystyle f_{lm}} \, y_{lm}(\omega)
\end{align}
where the $y_{lm}(\omega)$ are the spherical harmonics that provide an orthonormal basis for $L_2(\S2)$ and $\langle \, , \rangle$ refers to the standard $L_2$ inner product.
For concreteness, we will work with the Legendre spherical harmonics given by~\cite{Freeden1998}
\begin{align}
  y_{lm}(\omega) = y_{lm}(\theta,\phi) = C_{lm} \, e^{i m \phi} \, P_{lm}(\cos{\theta})
\end{align}
where the $P_{lm}(\cdot)$ are associated Legendre functions and $C_{lm}$ is a constant so that the functions are orthonormal.

The triangular structure of the index set in Eq.~\ref{eq:spherical_harmonics:expansion} results from the fact that the $y_{lm}(\omega)$ are eigenfunctions of the Laplace-Beltrami operator on $S^2$.
We will denote the $(2l+1)$-dimensional space spanned by all spherical harmonics in band $l$ by $\mathcal{H}_{l}(\S2)$ and those spanned by the $y_{lm}(\omega)$ up to a maximum degree $L$ by $\mathcal{H}_{\leq L}(S^2)$, i.e. $\mathcal{H}_{\leq L}(S^2) = \bigoplus_{l=0 \cdots L} \mathcal{H}_l(S^2)$.
The $\mathcal{H}_l(S^2)$ are rotation invariant, i.e for $f \in \mathcal{H}_l(S^2)$ one has $R^* f \in \mathcal{H}_l(S^2)$ for any $R \in \mathrm{SO}(3)$.
The rotation is implemented in the spherical harmonics domain by Wigner-D matrices $W_{lm}^{m'}(R)$ that map for each $\mathcal{H}_l(S^2)$ the coefficients $f_{lm}$ to those of the rotated signal.
We refer, for example, to the book by Freeden~\cite{Freeden1998} for more details on spherical harmonics and $L_2(S^2)$.

\subsection{Scalable Reproducing Kernel Frames for $\mathcal{H}_{\leq L}(S^2)$}
\label{sec:wavelets:scalable_rk_frame}

Before turning to scalable reproducing kernel frames on the sphere, we introduce the concept.
It generalizes tight sampling expansions formulated in the setting of using reproducing kernel Hilbert spaces.

\subsubsection{Scalable reproducing kernel frames}
We begin with the definition.

\begin{definition}
  \label{def:scalable_rk_frame}
  Let $\mathcal{H}_k(\M)$ be a reproducing kernel Hilbert space defined over a domain $\M$ with reproducing kernel $k_x(y) \equiv k(x,y)$.
  A \emph{scalable reproducing kernel frame} defined over the set $\{ ( w_k \, , \, k_{\lambda_k}(x) ) \}_{k \in \mathcal{K}}$ with positive weights $w_k \in \R^+$, locations $\lambda_k \in \mathcal{\M}$, and index set $\mathcal{I}$ is a frame for $\mathcal{H}_k(\M)$ such that
  \begin{align}
    \label{eq:scalable_rk_frame}
    f(x)
    = \sum_{k \in \mathcal{K}} \big\langle f(y) , k_{\lambda_k}(y) \big\rangle \, w_k \, k_{\lambda_k}(x)
    = \sum_{k \in \mathcal{K}} f(\lambda_k) \, w_k \, k_{\lambda_k}(x)  .
  \end{align}
  for all $f \in \mathcal{H}(\M)$.
\end{definition}

Compared to a Parseval tight reproducing kernel frame, such as the $\mathrm{sinc}$-basis expansion in the classical Shannon-Whittaker-Kotelnikov sampling theorem, the weights $w_k$ in the above definition provide additional flexibility. For instance, they can compensate for a lack of equi-distribution of the locations $\lambda_k$.
This is particularly useful when equi-spaced $\lambda_k$ are difficult to obtain or do not exist, such as when $\M$ is a nontrivial manifold.
In contrast to a general irregular sampling theorem, e.g.~\cite{Benedetto1992,Higgins1994}, the dual frame functions (or reconstruction kernels) $w_k \, k_{\lambda_k}(x)$ in Eq.~\ref{eq:scalable_rk_frame} differ from the reproducing kernels, $k_{\lambda_k}(x)$, thereby only by the scalar weight $w_k$.
Hence, they are immediately available and no expensive computation is necessary to obtain them.
This becomes particularly clear by introducing the weighted reproducing kernel $\bar{k}_{\lambda_k}(x) \equiv \sqrt{w_k} \, k_{\lambda_k}(x)$.
Eq.~\ref{eq:scalable_rk_frame} can then be written in the symmetric form
\begin{align}
  \label{eq:scalable_rk_frame:symmetric}
  f(x)
  = \sum_{i \in \mathcal{I}} \big\langle f(y) , \bar{k}_{\lambda_k}(y) \big\rangle \, \bar{k}_{\lambda_k}(x)
  = \sum_{i \in \mathcal{I}} \sqrt{w_k} \, f(\lambda_{k}) \, \bar{k}_{\lambda_k}(x)
\end{align}
that affords most of the practical advantages of an orthonormal sampling theorem yet has more flexibility through the $w_k$.

The weights $w_k$ can also be understood as scaling parameters for the frame vectors $k_{\lambda_k}(x)$.
Eq.~\ref{eq:scalable_rk_frame} is then a scalable frame in the sense recently introduced by Kutyniok, Okoudjou and co-workers~\cite{Kutyniok2013,Okoudjou2015,Chen2015}.
We borrow the nomenclature in Def.~\ref{def:scalable_rk_frame} from this connection.
For notational simplicity, we will in the following often identify a scalable reproducing kernel frame with its generating set, i.e. say that $\{ (\lambda_k , w_k) \}_{i \in \mathcal{I}}$ is the frame, or even more concisely $( \Lambda, w )$ where $\Lambda = \{ \lambda_k \}_{i \in \mathcal{I}}$ and $w = \{ w_k \}_{i \in \mathcal{I}}$.

\subsubsection{Scalable reproducing kernel frames $\mathcal{H}_{\leq L_j}(S^2)$}
For the construction of the discrete wavelets, we require scalable reproducing kernel frames for the spaces $\mathcal{H}_{\leq L_j}(S^2)$, e.g. with $L_j=2^j-1$.
For $L_j < \infty$, $\mathcal{H}_{\leq L_j}$ is finite dimensional and hence a reproducing kernel Hilbert space.
Its reproducing kernel is
\begin{align}
  \label{eq:HleqL:rk}
  k(\omega,\eta)
  = \sum_{l=0}^{L} \sum_{m=-l}^l y_{lm}^*(\omega) \, y_{lm}(\eta)
  = \sum_{l=0}^L \frac{2l+1}{4\pi} P_l(\omega \cdot \eta)
\end{align}
where the right hand side is a consequence of the spherical harmonics addition theorem.
A scalable reproducing kernel frame $(\Lambda_j,w_j)$ for $\mathcal{H}_{\leq L_j}(S^2)$ then consists of locations $\lambda_{jk} \in S^2$ and associated weights $w_{jk}$ such that every $L_j$-bandlimited function can be written as in Def.~\ref{def:scalable_rk_frame}.
Using the spherical harmonics representation in Eq.~\ref{eq:HleqL:rk}, it follows from a straightforward calculation that the frame can equivalently be characterized by
\begin{align}
  \label{eq:HleqL:scalable_frame_condition}
   \sum_{k \in \mathcal{K}}  w_{jk} \, y_{lm}^*(\lambda_{jk}) \, y_{l'm'}(\lambda_{jk}) = \delta_{l l'} \, \delta_{m m'} \quad , \quad l,l' \leq L .
\end{align}
Eq.~\ref{eq:HleqL:scalable_frame_condition} can be seen as a perfect reconstruction condition in terms of the $w_k$ and $\lambda_k$.
For an ideal, tight frame one has $w_k = 4 \pi / \vert \Lambda_j \vert$.

An important theoretical and practical question is the existence of scalable reproducing kernel frames for the spaces $\mathcal{H}_{\leq L_j}(\S2)$.
An answer thereby depends on the cardinality of the frame, i.e. its redundancy, and existing results in the literature strongly suggest that an increasing redundancy simplifies the problem.
From a practical point of view, however, one is interested in frames with a small redundancy.
At the moment, we are not able to guarantee the existence of scalable reproducing kernel frames for $\mathcal{H}_{\leq L_j}(\S2)$, at least for $L_j>1$, independent of the redundancy.
In this respect the situation is similar to those for related problems such as extremal points on the sphere~\cite{Womersley2001a,Sloan2004} and spherical $t$-designs, e.g.~\cite{Brauchart2015,Womersley2017}.
Following the approach taken for these in the literature~\cite{Hardin1996,Womersley2001a,Sloan2004,Sloan2009,Graf2011a,Womersley2017}, we also find weights and locations that provide scalable reproducing kernel frames using numerical optimization.
The details of the numerical construction are presented in Appendix~\ref{sec:scalable_rk_frames:optimization}.

\begin{figure}[t]
  \includegraphics[width=\textwidth]{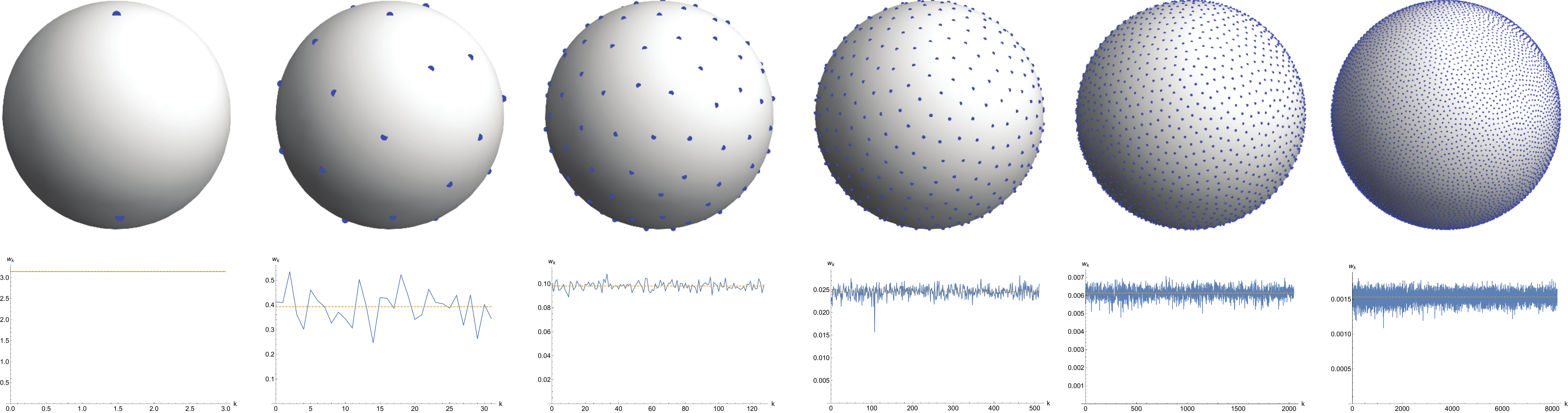}
  \caption{Locations $\lambda_{jk}$ (first row) and weights $w_{jk}$ for levels $j=1$, where the locations are given by the vertices of the tetrahedron, to $j=6$, where there are 8192 ones. The orange line in the plot for the weights indicates the ideal, uniform value $w_k = 4\pi / \vert \Lambda_j \vert$, i.e. those obtained for a tight reproducing kernel frame. A visualization of $w_{jk}$ as a function of $\lambda_{jk}$ is shown in Fig.~\ref{fig:lambdas_weighted}.}
  \label{fig:lambdas_weights}
\end{figure}

\subsubsection{Scalable reproducing kernel frames $\mathcal{H}_{\leq L_j}(S^2)$ with $L_j = 2^j - 1$}
To provide some insight into the results of the numerical construction, we consider the special case $L_j = 2^j - 1$.
This choice will also be used in the remainder of the paper whenever a value for $L_j$ needs to be fixed.

For the first two levels, corresponding to $L_{j=0}=0$ and $L_{j=1}=1$, the scalable reproducing kernel frames are defined analytically as the North Pole and the vertices of the tetrahedron.
In both cases one has an orthonormal reproducing kernel basis, i.e. there is no redundancy and $w_{jk} = 4\pi / \vert \Lambda_j \vert$.
The remaining levels all have redundancy $2$.
Our numerical experiments strongly indicate that for $L > 1$ no non-redundant scalable reproducing kernel frames exist and we also believe that for $L \gg 1$ a redundancy of $2$ is optimal.

Motivated by the dyadic grids used for wavelets in Euclidean spaces, the locations $\lambda_j$ are also chosen to be nested so that when $\lambda_j$ is part of the generating set for $\mathcal{H}_{\leq L_j}(\S2)$ then this is also true for all $\mathcal{H}_{\leq L_{j'}}(\S2)$ with $j' > j$, i.e. $\Lambda_j \subset \Lambda_{j'}$.
Note that for a redundancy of $2$ also the dimension of the resulting sequence of spaces $\mathcal{H}_{\leq L_j}(\S2)$ is consistent with the dyadic grids used in Euclidean space since the number of locations, $\vert \Lambda_j \vert = 2 (L_j + 1)^2 = 2^{2j+1}$, quadruples from level to level.
This is also the principal reasons for the choice $L_j = 2^j - 1$.
In contrast to the three different wavelets ones has for non-standard tensor product wavelets in the plane, however, we only have one fully isotropic wavelet.

Examples of our scalable reproducing kernel frames are shown in Fig.~\ref{fig:lambdas_weights}.
Fig.~\ref{fig:lambdas_weighted} displays the weights $w_{jk}$ as a function of the locations $\lambda_{jk}$.
The cardinality and the average distance $\bar{d}_j^{\mathrm{opt}}$ of the optimized points as well as their mesh norm $h_{\Lambda_j}$ for different levels is presented in the following table:
\vspace{1em}
\begin{center}
\newcolumntype{C}{>{\centering\arraybackslash}X}
\begin{tabularx}{0.9\textwidth}{ C|C|C|C|C|C }
  \hline
  $j$ & $L=2^j-1$ & $\vert \Lambda_j \vert$ & $\bar{d}_j^{\, \mathrm{opt}}$ & $\bar{d}_j^{\, \mathrm{avg}}$ & $h_{\Lambda_j}$
  \\[2pt]
  \hline
  $2$ & $3$ & $32$ & $0.7020$ & $0.7071$ & $0.4826$
  \\[2pt]
  $3$ & $7$ & $128$ & $0.3500$ & $0.3536$ & $0.2445$
  \\[2pt]
  $4$ & $15$ & $512$ & $0.1736$ & $0.1768$ & $0.1248$
  \\[2pt]
  $5$ & $31$ & $2048$ & $0.0866$ & $0.0884$ & $0.0615$
  \\[2pt]
  $6$ & $63$ & $8192$ & $0.0432$ & $0.0442$ & $0.0310$
  \\ \hline
\end{tabularx}
\end{center}
\vspace{1em}
The second but last column in the table is an idealized average distance that would be obtained when all locations have the same area associated with them and this is over a circular neighborhood, i.e.
\begin{align}
  \bar{d}_j^{\, \mathrm{avg}} = 2 \, \sqrt{\frac{4 \pi r^2}{\vert \Lambda_j \vert \pi}} = 4 \, \vert \Lambda_j \vert^{-j/2} .
\end{align}
While a circular neighborhood is geometrically unattainable, it is reasonable idealization in light of known results for optimal finite frames, cf.~\cite{Benedetto2003}.
The good agreement of the last two columns in the table is a clear indication of the well distributedness of the locations, which is also evident from a visual inspection of the plots in Fig.~\ref{fig:lambdas_weights}.

\begin{figure}
  \centering
  \includegraphics[trim={2.5cm 4.9cm 0 2.75cm}, clip, width=0.25\textwidth]{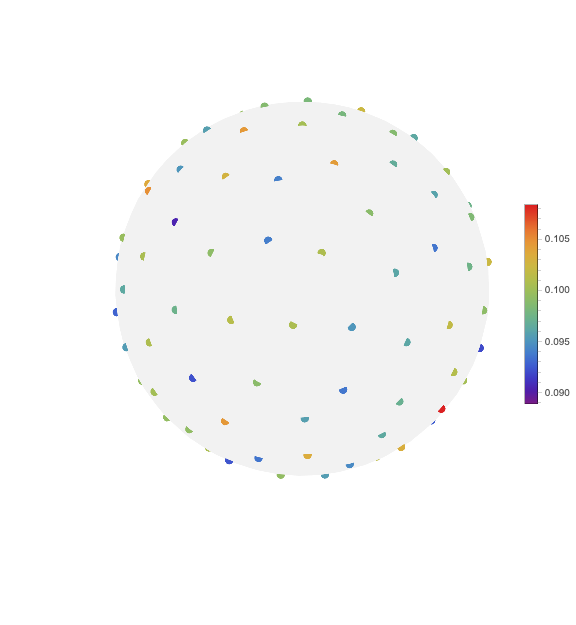}
  \includegraphics[trim={2.5cm 4.9cm 0 2.75cm}, clip, width=0.25\textwidth]{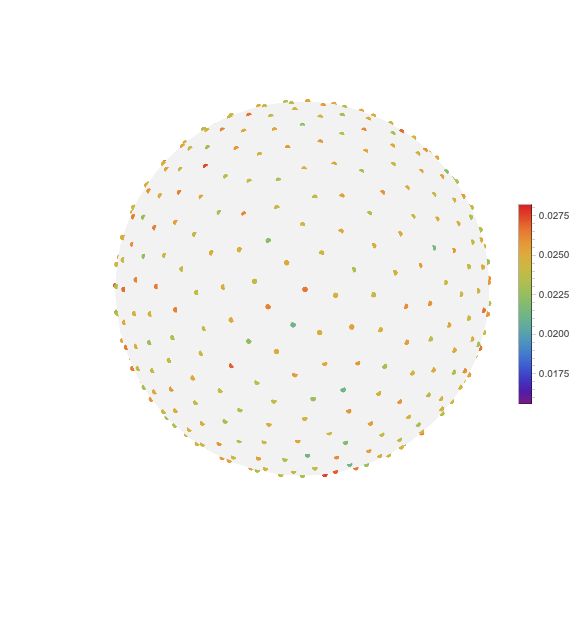}
  \includegraphics[trim={2.5cm 4.9cm 0 2.75cm}, clip, width=0.25\textwidth]{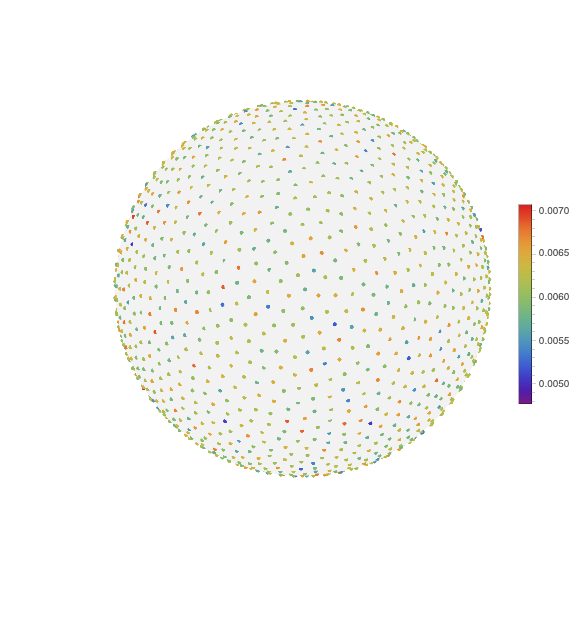}
  \caption{Locations and color-coded weights for $j=3,4,5$. The plots demonstrate that the weights compensate for the lack of equi-distribution of points, with large weights where the local density of points is low and small ones where the density is high.}
  \label{fig:lambdas_weighted}
\end{figure}

\begin{remark}[Connection to interpolatory quadrature rules]
\label{remark:wavelets_quadrature}
Our scalable reproducing kernel frames are closely related to interpolatory quadrature rules.
Specializing Eq.~\ref{eq:scalable_rk_frame} to $H_{\leq L}(S^2)$ and using Eq.~\ref{eq:HleqL:rk} we obtain for $f \in H_{\leq L}(S^2)$
\begin{subequations}
\begin{align}
  \int_{S^2} f(\omega) \, d\omega
  &= \sum_{i \in \mathcal{K}} f(\lambda_k) \, w_k \, \sum_{l,m} y_{lm}^*(\lambda_k) \int_{S^2} y_{lm}(\omega) \, d\omega  .
\end{align}
Since only $y_{00}(\omega)$ has a non-vanishing integral, which is equal to $\sqrt{4\pi}$, and $y_{00}(\omega) = 1/\sqrt{4\pi}$ we have
\begin{align}
   \int_{S^2} f(\omega) \, d\omega
   &= \sum_{i \in \mathcal{K}} w_k \, f(\lambda_k) .
\end{align}
\end{subequations}
For interpolatory quadrature rules~\cite[Sec.~4.3,~4.4]{Hesse2010a}, including those with extremal points, one has, typically, a biorthogonal, non-redundant reproducing kernel basis expansions.
The quadrature weights are then given by $w_k = \tilde{k}_{00}^i$, where the $\tilde{k}_{lm}^i$ are the spherical harmonics coefficients for the dual kernel functions (or reconstruction kernels) satisfying the biorthogonality (and interpolation) condition $\langle k_{\lambda_k}(\omega) , \tilde{k}_j(\omega) \rangle = \delta_{ij}$.
Our construction is hence considerably more stringent than interpolatory quadrature rules in that we require, up to the scaling given by the $w_k$, a Parseval tight frame.
To satisfy these requirements, we rely on redundancy.
As is apparent from Fig.~\ref{fig:lambdas_weights}, our quadrature weights are always positive, as desired~\cite{Hesse2010a}, and close to the optimal value of $4\pi / N$.
Similar to the situation for extremal points~\cite{Womersley2001a,Sloan2004}, we currently do not have a proof that guarantees the positivity.
\end{remark}

\subsection{A Discrete Wavelet Frame for $L_2(S^2)$}
\label{sec:wavelets:discrete}

The discrete wavelet frame is obtained from a set of mother wavelets $\{ \psi_j(\omega) \}_{j=0 \cdots}$, defined at the North Pole, that are rotated to the locations of scalable reproducing kernel frames $(\Lambda_j , w_j)$.
Since the wavelets we consider are isotropic, the mother wavelets $\psi_{j}(\omega)$ at the pole are given by
\begin{align}
  \label{eq:mother_wavelet}
  \psi_j(\omega) = \sum_{l} \kappa_l^j \, y_{l0}(\omega) .
\end{align}
The $L_j$-bandlimited window coefficients $\kappa_l^j$ are carefully chosen to ensure that the $\psi_j(\omega)$ induce a tight frame for $L_2(S^2)$, see Theorem~\ref{thm:wavelets} below, and that they are well localized in the spherical harmonics and spatial domains.
As in the Euclidean case, the mother wavelets are complemented by father scaling functions $\phi_j(\omega)$ to represent the low frequency parts of a signal.
The father scaling functions are also isotropic and defined at the North Pole through window coefficients $\bar{\kappa}_{l}^j$, analogous to Eq.~\ref{eq:mother_wavelet}. The $\bar{\kappa}_{l}^j$ are chosen so that the scaling functions on level $j$ together with the wavelets on the same and all finer levels lead to a representation for $L_2(S^2)$.
To simplify notation, we will locate the scaling functions on level $j=-1$, i.e. $\psi_{-1}(\omega) \equiv \phi_0(\omega)$ (in applications it is sometimes advantageous to use some level $j'$ as coarsest one and then one relabels the levels so that $j \rightarrow j - j'$).

Given the mother wavelet centered at the North Pole, it has to be rotated to cover the entire sphere.
Equivalently, we used so far only $m=0$ in the spherical harmonics frequency domain, cf. Eq.~\ref{eq:mother_wavelet}, and we need to populate the entire triangular $(l,m)$ parameter space.
 For the mother wavelets $\psi_j(\omega)$ with coefficients $\kappa_l^j$, the rotation in spherical harmonics space through the Wigner-D matrices is given by $W_{lm}^{0}(R_{\eta}) = \sqrt{4\pi / 2l+1} \, y_{lm}(\eta)$, where $R_{\eta}$ is the rotation from the North Pole to $\eta \in S^2$.

To obtain a tight frame, we will use as locations the $\lambda_{jk} \in \Lambda_j$ of a scalable reproducing kernel frame $\{ (\lambda_{jk} \, , \, w_{jk}) \}_{k \in \mathcal{K}_j}$ for $\mathcal{H}_{L_j}(\S2)$ and the square roots of the weights will provide weighting factors to obtain a symmetric representation where primary and dual functions coincide, analogous to Eq.~\ref{eq:scalable_rk_frame:symmetric}.
We define the discrete spherical wavelets thus defined as
\begin{align}
  \label{eq:wavelets:lambda_jk}
  \psi_{jk}(\omega)
  \equiv \sum_{l=0}^{L_j} \sum_{m=-l}^l \, \underbrace{\sqrt{\frac{4\pi}{2l+1}} \, \sqrt{w_{jk}} \, \kappa_l^j \, y_{lm}(\lambda_{jk})}_{\displaystyle \kappa_{lm}^{jk} \equiv \psi_{lm}^{jk}} \, y_{lm}(\omega)
\end{align}
where the $\kappa_{lm}^{jk} \equiv \psi_{lm}^{jk}$ denote the spherical harmonics coefficients of the $\psi_{jk}(\omega)$.
With Eq.~\ref{eq:wavelets:lambda_jk}, the main result of the present section is the following.

\begin{figure}
  \centering
  \includegraphics[width=0.7\textwidth]{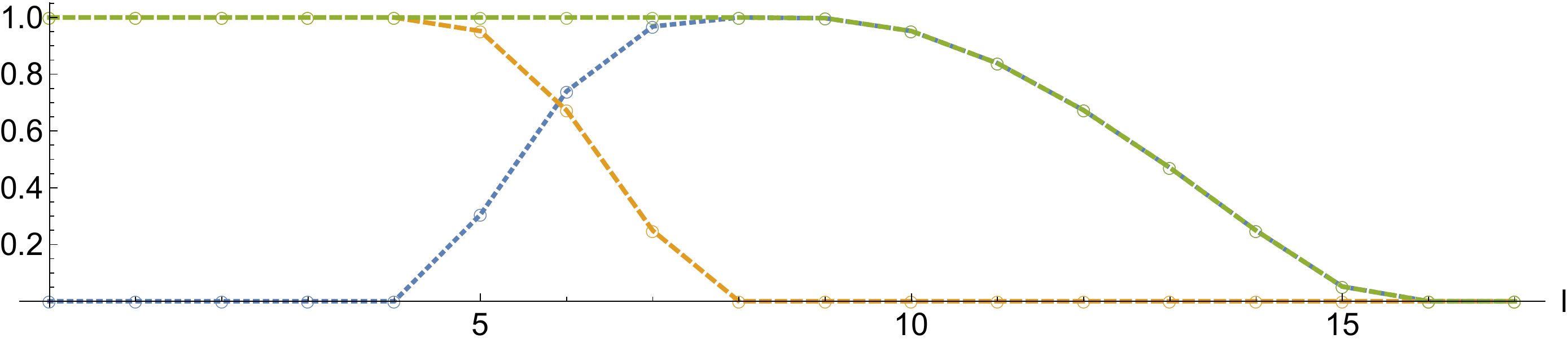}
  \caption{The difference between the scaling function windows $\bar{\kappa}_j^l$ on adjacent levels $j$ and $j+1$ (dashed) defines the wavelet windows $\kappa_j^l$ for level $j$ (dotted). Windows are normalized by $\sqrt{2l+1 / 4 \pi}$.}
  \label{fig:kappa_j_l_phi_psi}
\end{figure}

\begin{theorem}
  \label{thm:wavelets}
  Let the $L_j$-bandlimited window coefficients $\kappa_{l}^j$ satisfy the Calder{\'o}n admissibility condition
  \begin{align}
    \label{eq:calderon_condition}
   \forall l \ : \quad \frac{4\pi}{2l+1} \sum_{j=-1}^{\infty} \vert \kappa_{l}^j \vert^2 = 1 \quad
  \end{align}
  and let $\{ (\lambda_{jk} , w_{jk}) \}_{k \in \mathcal{K}_j}$ for $j = 0, 1 \cdots$ be a sequence of scalable reproducing kernel frames for the spaces $\mathcal{H}_{\leq L_j}(\S2)$.
  Then the wavelets in Eq.~\ref{eq:wavelets:lambda_jk} form a Parseval tight frame for $L_2(S^2)$.
\end{theorem}

For the proof of the theorem we will use the following lemma, which will also require again in Sec.~\ref{sec:forms}.

\begin{lemma}
  \label{prop:wavelets:calderon}
  Under the assumptions in Theorem~\ref{thm:wavelets},
  \begin{align}
    \label{eq:thm:wavelets:proof:5}
    \delta_{l l'} \, \delta_{m m'} &= \sum_{j=-1}^{\infty} \sum_{k \in \mathcal{K}_j} \psi_{lm}^{jk \, *} \, \psi_{l'm'}^{jk} \quad , \quad \forall l,l',m,m' .
  \end{align}
\end{lemma}

\begin{proof}
Using Eq.~\ref{eq:wavelets:lambda_jk} we can write
\begin{subequations}
\begin{align}
  \delta_{l l'} \, \delta_{m m'}
  &=
  \sum_{j=-1}^{\infty} \! \sum_{k \in \mathcal{K}_j} \Big(\sqrt{\frac{4\pi \, w_{jk}}{2l+1}} \, \kappa_l^j \, y_{lm}^*(\lambda_{jk}) \Big) \Big( \sqrt{\frac{4\pi \, w_{jk}}{2l'+1}} \, \kappa_{l'}^j \, y_{l'm'}(\lambda_{jk}) \Big)
  \\[3pt]
  &= \sum_{j=-1}^{\infty} \frac{4\pi}{\sqrt{2l+1} \sqrt{2l'+1}} \kappa_l^j \, \kappa_{l'}^j \sum_{k \in \mathcal{K}_j}  w_{jk} \, y_{lm}^*(\lambda_{jk}) \, y_{l'm'}(\lambda_{jk}) .
\end{align}
\end{subequations}
Since $\{ (w_{jk} , \lambda_{jk} ) \}_{k \in \mathcal{K}_j}$ forms a scalable reproducing kernel frame for every $\mathcal{H}_{\leq L_j}(\S2)$ the sum over $k$ equals $\delta_{l l'} \, \delta_{m m'}$ whenever the product $\kappa_l^j \kappa_{l'}^j$ is nonzero, see Eq.~\ref{eq:HleqL:scalable_frame_condition}.
Together with the Calder{\'o}n admissibility condition in Eq.~\ref{eq:calderon_condition} this implies Eq.~\ref{eq:thm:wavelets:proof:5}.
\end{proof}

\begin{figure}
  \centering
  \includegraphics[width=0.9\textwidth]{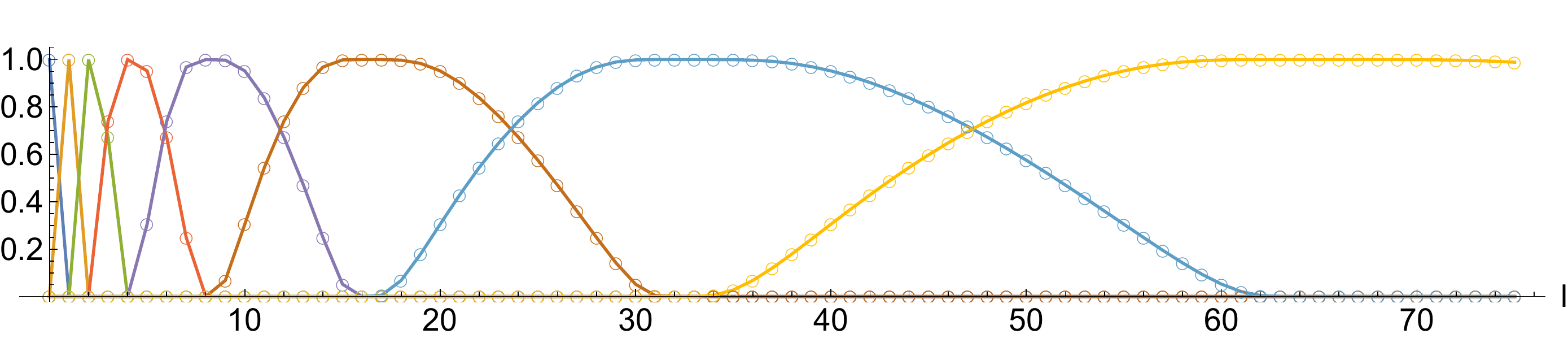}
  \caption{Windows $\kappa_j^l$ that define the scaling function on the coarsest level (with support at $l=0$) and the wavelet functions on subsequent ones. Windows are normalized by $\sqrt{2l+1 / 4 \pi}$.}
  \label{fig:kappa_j_l}
\end{figure}

\begin{proof}[Proof of Theorem~\ref{thm:wavelets}]
  Using Lemma~\ref{prop:wavelets:calderon}, the spherical harmonics representation of any $f \in L_2(\S2)$ in Eq.~\ref{eq:spherical_harmonics:expansion} can be written as
  \begin{subequations}
  \begin{align}
     f(\omega) &= \sum_{l=0}^{\infty} \sum_{m=-l}^l \sum_{l'=0}^{\infty} \sum_{m=-l'}^{l'} f_{lm} \left( \sum_{j=-1}^{\infty} \sum_{k \in \mathcal{K}_j} \psi_{lm}^{jk \, *} \, \psi_{l'm'}^{jk} \right) y_{l'm'}(\omega) .
  \end{align}
  Using linearity and re-arranging terms we obtain
  \begin{align}
    f(\omega)
    &= \Bigg\langle f(\eta) \, , \sum_{j=-1}^{\infty} \sum_{k \in \mathcal{K}_j} \sum_{l=0}^{L_j} \sum_{m=-l}^l \psi_{lm}^{jk \, *} \, y_{lm}^*(\eta) \, \sum_{l'=0}^{L_j} \sum_{m=-l}^l \psi_{l'm'}^{jk} \, y_{l'm'}(\omega) \Bigg\rangle
  \end{align}
  By the Parseval identity and the definition of the wavelet functions in Eq.~\ref{eq:wavelets:lambda_jk} this equals
  \begin{align}
    f(\omega) = \sum_{j=-1}^{\infty} \sum_{k \in \mathcal{K}_j} \big\langle f(\eta) , \psi_{jk}(\eta) \big\rangle \, \psi_{jk}(\omega) .
  \end{align}
  \end{subequations}
\end{proof}

The above theorem ensures that arbitrary signals $f \in L_2(S^2)$ can be represented using the discrete spherical wavelets in Eq.~\ref{eq:wavelets:lambda_jk} and that the representation affords many of the conveniences of an orthonormal basis, such as that primary and dual frame functions coincide and that Parseval's identity holds.

\begin{example}[Shannon wavelet]
  \label{ex:shannon_wavelet}
  The simplest example of wavelets satisfying our requirements are the spherical Shannon wavelets.
  Their scaling functions are given by
  \begin{align}
    \bar{\kappa}_l^j =
    \left\{
    \begin{array}{cc}
      1 & l < 2^{\lceil j-1/2 \rceil}
      \\[3pt]
      0 & \mathrm{otherwise}
    \end{array}
    \right.
  \end{align}
  where $\lceil \cdot \rceil$ refers to the ceiling operation.
  The wavelets are then $\smash{\kappa_l^j = \bar{\kappa}_l^{j+1} - \bar{\kappa}_l^j}$.
  As in the Euclidean case, because the windows defined by the $\smash{\bar{\kappa}_l^j}$ are not smooth in the spherical harmonics frequency domain, the wavelets suffer from a slow decay in the spatial domain.
\end{example}

\begin{example}
Window coefficients $\kappa_l^j$ that yield wavelets with fast decay in the spatial domain were constructed by McEwen, Durastani, and Wiaux~\cite{McEwen2016}.
The $\kappa_l^j$ in this case satisfy
\begin{align}
  \label{eq:supp_kappa_j_l}
  \mathrm{supp}(\kappa_l^j) = \big[ \lfloor 2^{j-1}+1 \rfloor , 2^{j+1}-1 \big] ,
\end{align}
where $\lfloor \cdot \rfloor$ refers to the flooring operation.
We thus have $L_j = 2^{j+1}-1$.
The smooth character of the $\kappa_l^j$ in $l$, cf. Fig.~\ref{fig:kappa_j_l}, ensures that the induced wavelets have quasi-exponential decay in the spatial domain~\cite{McEwen2016}.
Plots of the wavelet functions can be found in Fig.~\ref{fig:psis_spatial_profile}.
It follows from Eq.~\ref{eq:supp_kappa_j_l} that $\phi_{0}(\omega) = y_{00}(\omega)$.
For $j \geq 0$, the scaling functions again satisfy $\kappa_j^l = \bar{\kappa}_{k+1}^l - \bar{\kappa}_j^l$, see Fig.~\ref{fig:kappa_j_l_phi_psi}.
Unless mentioned otherwise, we will use the windows by McEwen and co-workers~\cite{McEwen2016} in the remainder of the paper.
\end{example}

\begin{figure}
  \includegraphics[width=\textwidth]{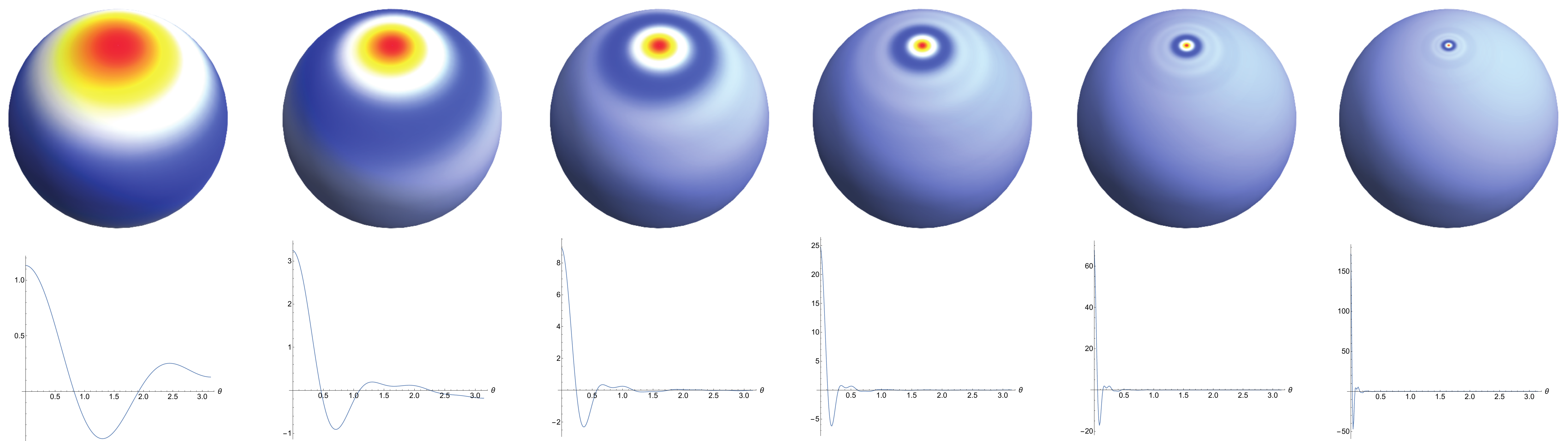}
  \caption{Wavelets $\psi_{j0}(\theta,\phi)$ and their profiles $\psi_{j0}(\theta,0)$ for $j=1$ to $j=6$ based on the windows from~\cite{McEwen2016}. The direct comparison to Fig.~\ref{fig:lambdas_weights} shows the correspondence between the effective support of the wavelets and the density of the point set over which they are supported.}
  \label{fig:psis_spatial_profile}
\end{figure}

\begin{remark}[Generalized multi-resolution structure]
The wavelet frame in Theorem~\ref{thm:wavelets} does not form a classical multi-resolution analysis~\cite{Mallat1989,Meyer1986}.
However, with $\kappa_j^l = \bar{\kappa}_{j+1}^l - \bar{\kappa_j}^l$, as in the two examples above, it is a generalized one~\cite{Merrill2018} in the sense of Baggett, Carey, Moran, and Ohring~\cite{Baggett1995}.
For our purposes, the essential difference is that in a generalized multi-resolution analysis there is no longer the requirement that the scaling functions form a Riesz basis for the multi-resolution spaces $V_j$.
This becomes necessary in our case because of the, in general, smooth decay of the wavelets in the spherical harmonics domain, cf. Fig.~\ref{fig:kappa_j_l_phi_psi}, while using bandlimitedness for discretization.
This causes the scaling functions to not be in the multi-resolution spaces $V_j = \mathcal{H}_{\leq L_{j-1}}(\S2)$.
The situation is in this respect analogous to those for pyramid schemes in Euclidean space~\cite{Unser2013}.

With $D_2$ being dyadic dilation, we have the following definition for a generalized, dyadic multi-resolution analysis on $S^2$:
\vspace{0.4em}
\begin{enumerate}[i.)]
  \setlength\itemsep{0.2em}
  \item $V_{j}(S^2) \subseteq V_{j+1}(S^2)$;
  \item $L_2(S^2) = \mathrm{Closure}\Big( \bigcup_{j=0}^{\infty} V_{j}(S^2) \Big)$ and $\bigcap V_j = \{ 0 \}$;
  \item $D_2(V_j) = V_{j+1}$;
  \item $V_0$ is invariant under the action of $\mathrm{SO}(3)$ .
\end{enumerate}
\vspace{0.4em}
Since $\kappa_j^l = \bar{\kappa}_{j+1}^l - \bar{\kappa}_j^l$, by the Calder{\'o}n condition in Eq.~\ref{eq:calderon_condition} the $\bar{\kappa}_j^l$ are unity for $l < L_{j-1}$.
The multi-resolution spaces $V_j(S^2)$ are thus $V_j(S^2) = \mathcal{H}_{L_{j-1}}(S^2)$ while the bandlimit of the scaling functions is $L_j$, so that they are not themselves in the spaces.
With this, properties i.), ii.) and iv.) above are easily verified and for iii.) one can exploit that the associated Legendre functions $P_{lm}(\cos{\theta})$ have finite Fourier series representations with bandlimit $l$.
The Shannon wavelets in Example~\ref{ex:shannon_wavelet} are, in fact, contained in the spaces $V_j$ but they fail to satisfy the dyadic translation condition required in the classical ones.
The above definition of a generalized multi-resolution analysis is also consistent with other definitions of multi-resolution structures for $S^2$ that have been proposed in the literature~\cite{Potts1996,Freeden2018}.

The generalized multi-resolution structure can also be understood from the point of view of sampling theory.
Our scaling functions are, in general, smoothed versions of the ideal sampling function, i.e. the Shannon scaling function in Example~\ref{ex:shannon_wavelet}.
The smooth decay of the defining filter taps $\kappa_l^j$ in $l$ leads to better spatial localization while still verifying the reproducing property for functions $f_j \in V_j(S^2)$,
\begin{align}
  \big\langle f_j(\omega) \, , \, \phi_{jk}(\omega) \big\rangle = f(\lambda_{jk}) .
\end{align}
This is in full analogy to the Euclidean case.
There one can also construct generalizations of the Shannon sampling expansion with reconstruction kernels with better spatial localization by defining them with a smooth tail in the frequency domain, see e.g.~\cite{Strohmer2005}.
\end{remark}

\begin{figure}
  \centering
  \includegraphics[width=0.8\textwidth]{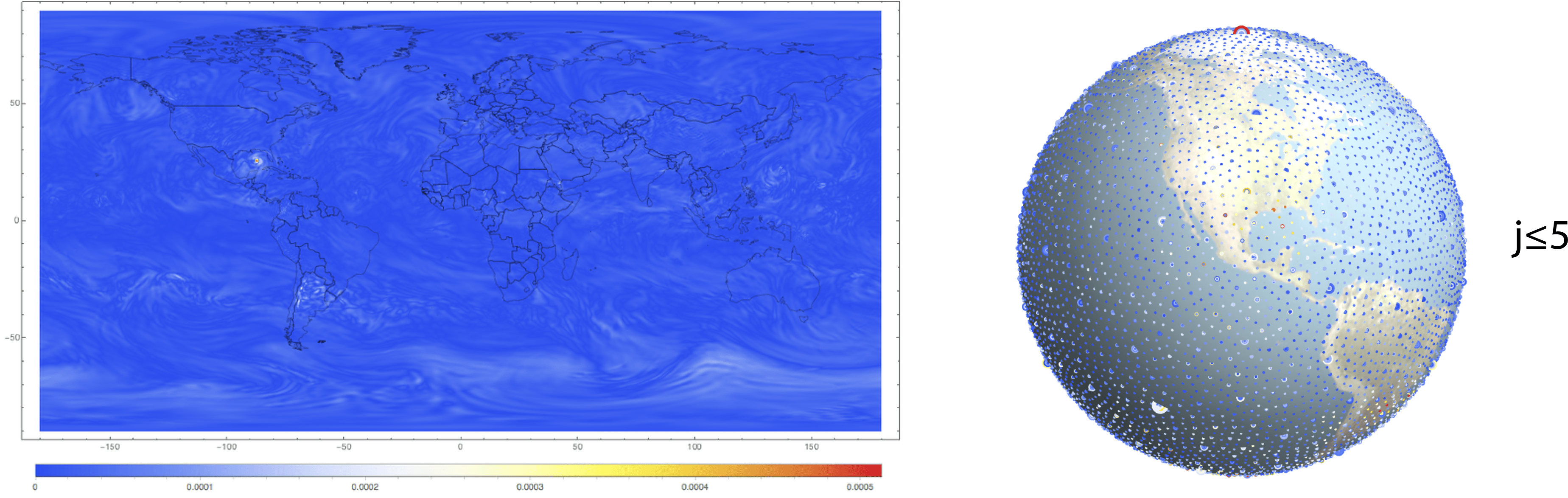}
  \caption{\emph{Left:} Visualization of potential vorticity (PV) for 29/08/2005, the day hurricane Katrina made landfall at the Gulf coast of the USA.
    \emph{Right:} Superposition of wavelet coefficients for all levels, with the level encoded in the size of the point and the magnitude of the coefficients through a temperature map, with blue corresponding to low values and red to large ones.
    Individual levels are shown in Fig.~\ref{fig:katrina_panel:1}.
    }
  \label{fig:katrina_panel:1}
\end{figure}

\begin{remark}[Connection to needlets]
\label{remark:wavelets:discrete:needlets}
A construction closely related to ours are the needlets by Narcowich, Petrushev, and Ward~\cite{Narcowich2006a}.
For these, first a scale discrete but in the spatial domain continuous wavelet frame is introduced and this is then discretized using a quadrature rule.
More precisely, Narcowich, Petrushev, and Ward observe that for the scale discrete expansion
\begin{align}
  \label{eq:needlets:semi_discrete}
  f(\omega)
  = \sum_{j=-1}^{\infty} \int_{S^2} \big\langle f(\xi) , \psi_j(\xi,\eta) \big\rangle_{\xi} \, \psi_j(\eta,\omega) \, d\eta
  = \sum_{j=-1}^{\infty} \int_{S^2} \bar{f}_j^{\psi}(\eta) \, \psi_j(\eta,\omega) \, d\eta
\end{align}
the coordinate function $\smash{\bar{f}_j^{\psi}(\eta)}$ as well as the $\smash{\psi_j(\eta,\omega)}$ are both $L_{j+1}$-bandlimited when this also holds for the window coefficients $\smash{\kappa_{l}^j}$.
Their product is thus in $\smash{\mathcal{H}_{L_{j+2}}(\S2)}$ and the reconstruction integral can be implemented with a finite quadrature rule for the space.
The existence of the quadratures is proved by the authors.
This provides the discrete needlet frame.
To the best of our knowledge, however, there exists no constructive algorithm to obtain the nodes and weights of the quadrature rules except for nonlinear optimization.

\begin{figure}
  \centering
  \includegraphics[width=0.95\textwidth]{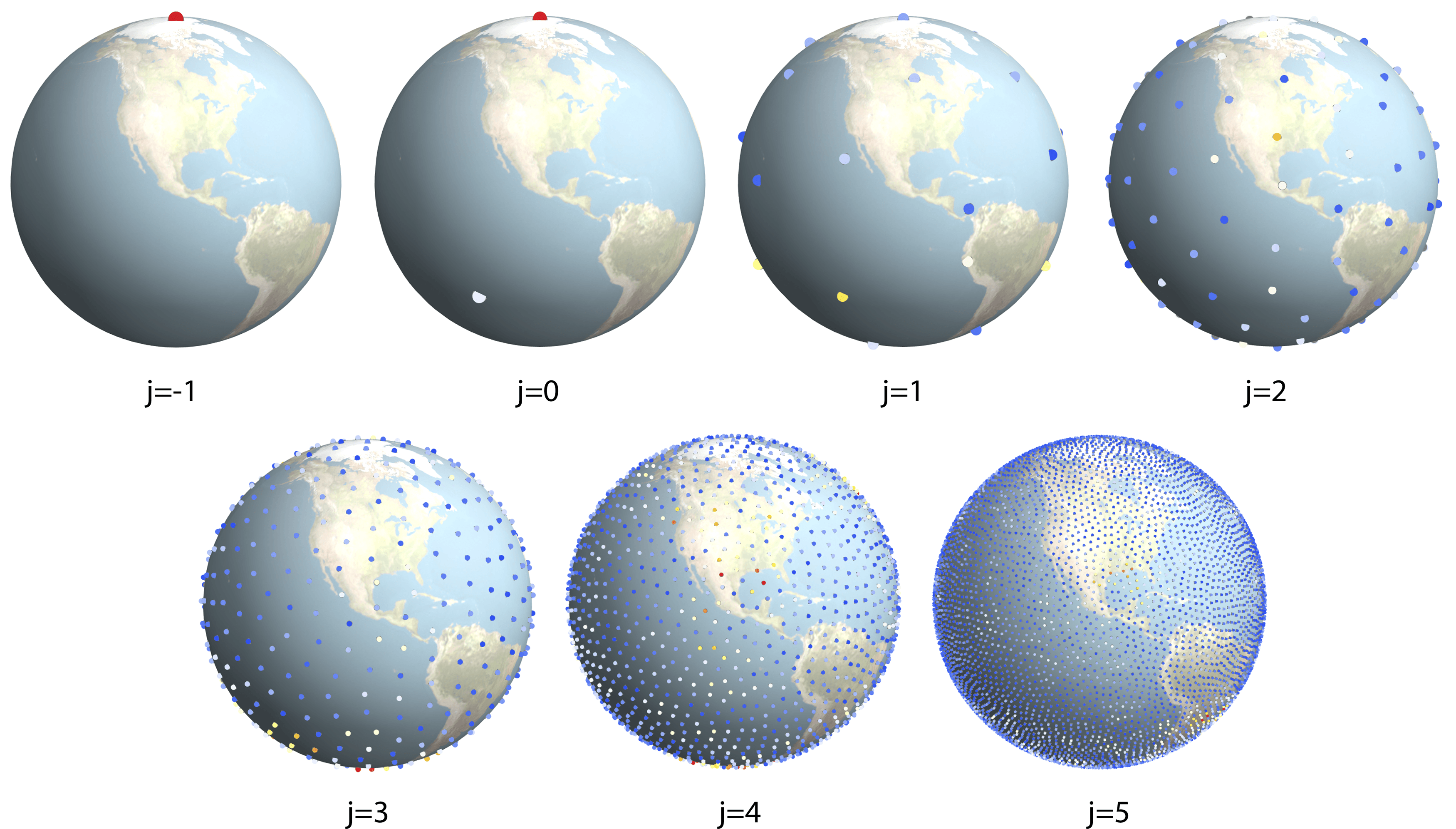}
  \caption{Visualization of wavelet coefficients for potential vorticity (PV) for 29/08/2005, the day hurricane Katrina made landfall at the gulf coast of the USA, see Fig.~\ref{fig:katrina_panel:1} for the spatial potential vorticity field.
  The coefficients are plotted at the locations of the basis functions with the magnitude encoded with a temperature map, with blue corresponding to low values and red to large ones.}
  \label{fig:katrina_panel:2}
\end{figure}

As an alternative to the quadrature rule in the original construction of needlets, the $L_j$-bandlimited function $\bar{f}_j^{\psi}(\eta)$ can be written in a scalable reproducing kernel frame, i.e.
\begin{align}
  f_j^{\psi}(\eta) = \sum_{\lambda_{jk} \in \Lambda_j} f_j^{\psi}(\lambda_k) \, w_k \, k_{\lambda_k}(\eta) .
\end{align}
Inserting into Eq.~\ref{eq:needlets:semi_discrete} then also yields a fully discrete wavelet frame, and this one is identical to those in Theorem~\ref{thm:wavelets}.
An alternative proof of Theorem~\ref{thm:wavelets} could hence proceeds in this way, and this is, in fact, the original one we developed.

Needlets and our construction are in most respects equivalent.
One aspect there this is not obvious is the redundancy of the resulting wavelet representations.
For our construction with scalable reproducing kernel frames of redundancy $2$, the number of wavelets on each levels is $\mathcal{O}( 2 L_{j+1}^2 )$.
With spherical $t$-designs, which provide the optimal discretization for needlets, numerical experiments suggest that their minimal cardinality is $1/2 (t^2 + 1)$, e.g.~\cite{Hardin1996,Sloan2009,Graf2011a}.
Hence, also for needlets one has $\mathcal{O}( 2 L_{j+1}^2)$ wavelets per level.
Both constructions thus also yield the same redundancy.
This suggests that spherical $t$-designs for $\mathcal{H}_{\leq 2 L}(\S2)$ correspond to scalable reproducing kernel frames for $\mathcal{H}_{\leq L}(\S2)$ and we verified this numerically for some examples from~\cite{Womersley2017}.
This is also supported by our experimental observation that for large $L$ a redundancy of $2$ is required to obtain scalable reproducing kernel frames and that non-redundant representations are only possible with $L \leq 1$, analogous to the fact that tight spherical designs only exist for $L = 1,2,3,5$.
We believe that the connection between scalable reproducing kernel bases and spherical $t$-designs deserves further attention in future work.
\end{remark}

\begin{remark}
The wavelets defined in Eq.~\ref{eq:wavelets:lambda_jk} are isotropic.
However, because of the separability of the longitudinal component at the North Pole, the construction is naturally extended to anisotropic ones by using window coefficients $\smash{\kappa_{lm}^{j,t} = \kappa_{l}^j \, \beta_{m}^{j,t}}$, where $t$ is a directional orientation parameter in $T S^2$.
The required $\smash{\beta_{m}^{j,t}}$ are available in the literature on steerable and polar wavelets, e.g.~\cite{Unser2010} and~\cite{Lessig2018z}, and allow for flexible angular localization.
In contrast to related constructions that have appeared previously~\cite{Starck2006,McEwen2015,Chan2015}, these anisotropic spherical wavelets thereby still form a discrete Parseval tight frame, although with a larger redundancy that in the isotropic case.
In analogy to the situation in the plane~\cite{Unser2010,Lessig2018z}, a suitable choice of the $\smash{\beta_{m}^{j,t}}$ yields ridgelet- and curvelet-like wavelets~\cite{Candes1999a,Candes2005b,Do2003,Labate2005}.
These could be beneficial for the analysis of highly anisotropic features like global circulation patterns.
Improved sparsity could, however, only be obtained with an anisotropic grid of locations whose construction is currently unclear to us.
To our knowledge, there are currently also no approximation-theoretic results for anisotropic wavelets on the sphere.
We leave a thorough investigation of the directional case thus to future work.
\end{remark}

\subsection{Numerical Example}
\label{eq:wavelets:experiments}

To exemplify the practical properties of the wave\-lets introduced above, we projected the potential vorticity (PV) of 29/08/2005, the day hurricane Katrina hit the Gulf coast of the USA, into the frame.
The original spatial vorticity field (from the ERA5 reanalysis data set~\cite{ERA5}) is shown on the left in Fig.~\ref{fig:katrina_panel:1} and a visualization of the superposition of the wavelet coefficients on all levels on the right.
Per level coefficients are presented in Fig.~\ref{fig:katrina_panel:2}.

In the plots, hurricane Katrina is clearly discernible in coefficient space, demonstrating the spatial locality that is provided by the wavelets.
At the same time, each level still has a clear frequency localization, cf. Fig.~\ref{fig:kappa_j_l}.
The large coefficient at the North Pole for coarse levels is the wavelet representation of the polar vortex.
Further numerical results will be presented in Sec.~\ref{sec:shallow}.

%
%


\section{A Local Spectral Exterior Calculus for $S^2$}
\label{sec:forms}

In the present section, we introduce spherical differential form wavelet $\psi_{jk}^{\fdeg,\nu}$ and the local spectral exterior calculus $\Psiec$ defined on them.
Towards this end, we will first discuss a spectral exterior calculus for $S^2$ that, together with the scalar spherical wavelets introduced in the last section, provides the basis for the differential form wavelets.
We will begin by recalling some basic facts about exterior calculus and fixing notation.
For a thorough introduction we refer to the literature, e.g.~\cite{Marsden2004} and~\cite{Frankel2011}.

\subsection{Notation}
\label{sec:forms:notation}

A differential $\fdeg$-form is a covariant, anti-symmmetric tensor of rank $\fdeg$.
Geometrically, it can be understood as an object that is naturally integrated over an $\fdeg$-dimensional (sub-)manifold~\cite{Frankel2011}.
The spaces of differential $\fdeg$-forms on $S^2$ will be denoted as $\Omega^{\fdeg}(\S2)$ with $\Omega^0(\S2) \cong \mathcal{F}(\S2)$, i.e. the space of functions, and $\Omega^{\fdeg}(\S2) = \varnothing$ for $\fdeg < 0$ and $\fdeg > 2$.
The coordinate expressions for differential forms on $\S2$ in latitude-longitude $(\theta,\phi)$-coordinates are
\begin{subequations}
\begin{align}
  f(\omega) \in \, \, & \Omega^0(\S2) \cong \mathcal{F}(\S2)
  \\[3pt]
  \alpha(\omega) = \alpha_{\theta}(\omega) \, \dd \theta + \alpha_{\phi}(\omega) \, \dd \phi \in \, \, & \Omega^1(\S2)
  \\[3pt]
  \gamma(\omega) = \gamma_{\theta,\phi}(\omega) \, \dd \theta \wedge \dd \phi \in \ \, & \Omega^2(\S2) .
\end{align}
\end{subequations}
The form basis functions $\mathrm{d}\theta$, $\mathrm{d}\phi$ are the biorthogonal duals to the vector basis functions $\partial / \partial \theta$ and $\partial / \partial \phi$ induced by the coordinate chart (note that we will work with the unnormalized basis functions).
The multiplication on differential forms is the anti-symmetric wedge product $\wedge : \Omega^{\fdeg}(\S2) \times \Omega^l(\S2) \to \Omega^{{\fdeg}+l}(\S2)$, which can be seen as a generalization of the cross product in $\R^3$, and it turns the spaces $\Omega^{\fdeg}(\S2)$ into a graded algebra.
The exterior derivative $\dd : \Omega^{\fdeg}(\S2) \to \Omega^{\fdeg+1}(\S2)$ is the natural derivation acting on differential forms satisfying, among other things, a Leibniz rule and $\dd \cdot \dd = 0$.
It is also the covariant form of the usual differential operators gradient, curl, and divergence.
A differential form $\alpha$ with $\dd \alpha = 0$ is said to be closed and if $\alpha = \dd \beta$ for some $\beta$ then it is exact.

The foregoing concepts are metric-independent.
Using the Riemannian structure on $\S2$ induced by $\R^3$
we can introduce the Hodge dual $\star : \Omega^{\fdeg}(\S2) \to \Omega^{n-{\fdeg}}(\S2)$.
In coordinates it is
\begin{subequations}
\begin{align}
  \star f &= f \, \dd \theta \wedge \dd \phi
  \\[3pt]
  \star \big( \alpha_\theta \, \dd \theta + \alpha_{\phi} \, \dd \phi \big) &= -\alpha_\phi \, \dd \theta + \alpha_{\theta} \, \dd \phi
  \\[3pt]
  \star \big( \beta_{\theta,\phi} \, \dd \theta \wedge \dd \phi \big) &= \beta_{\theta,\phi} .
\end{align}
\end{subequations}
The Hodge dual induces an $L_2$ inner product on $\Omega^{\fdeg}(\S2)$ by
\begin{align}
  \label{eq:forms:inner_product}
  \langle \! \langle \alpha , \beta \rangle \! \rangle = \int_{\S2} \alpha \wedge \star \beta
\end{align}
with the integral on the right hand side being well defined since $\alpha \wedge \star \beta \in \Omega^2(\S2)$.
The adjoint of the exterior derivative under the above inner product is the co-differential $\delta : \Omega^{\fdeg+1}(\S2) \to \Omega^{\fdeg}(\S2)$, i.e. $\llangle \dd \alpha \, , \, \beta \rrangle = \llangle \alpha \, , \, \delta \beta \rrangle$.
It can also be expressed using the Hodge dual as $\delta = \star \, \dd \, \star$.
A differential form $\alpha$ with $\delta \alpha = 0$ is said to be co-closed and if $\alpha = \delta \beta$ for some $\beta$ then it is co-exact.

We will frequently make use of the Hodge-Helmholtz decomposition that splits $\Omega^{\fdeg}(\S2)$ into three orthogonal parts
\begin{align}
  \label{eq:hodge_helmholtz}
  \Omega^{\fdeg}(\S2) = \Omega_{\dd}^{\fdeg}(\S2) \, \bigoplus \, \Omega_{\delta}^{\fdeg}(\S2) \, \bigoplus \, \Omega_{h}^{\fdeg}(S^2)
\end{align}
where $\Omega_{\dd}^{\fdeg}(\S2)$ is the space of exact $\fdeg$-forms, $\Omega_{\delta}^{\fdeg}(\S2)$ those of co-exact ones, and $\Omega_{h}^{\fdeg}(S^2)$ are the harmonic forms, i.e. those in the kernel of the Laplace-Beltrami operator $\Delta = \delta \, \dd + \dd \, \delta$.
For $\fdeg=0$ and $\fdeg=2$, the harmonic forms are exactly the constants and $\Omega_{h}^{1}(S^2) = \emptyset$.
Any $\gamma \in \Omega^{\fdeg}(\S2)$ can thus be written as $\gamma = \dd \alpha + \delta \beta + \zeta$ where $\zeta$ is harmonic.
One of the important properties of Eq.~\ref{eq:hodge_helmholtz} is that it characterizes the domain and image of the exterior derivative $\dd$.
With the Hodge-Helmholtz decomposition, the Hodge dual can also be characterized more precisely as $\star : \Omega_{\nu}^{\fdeg}(\S2) \to \Omega_{\bar{\nu}}^{n-\fdeg}(\S2)$ where $\nu \in \{ \dd  , \delta , h \}$ and $\bar{\dd} = \delta$, $\bar{\delta} = \dd$, and $\bar{h} = h$.

The metric also allows one to introduce the musical isomorphisms that identify vector fields and $1$-forms,
\begin{subequations}
\begin{align}
  & v^{\flat} \in \Omega^1(\S2) \ , \ v \in \mathfrak{X}(\S2)
  \\[3pt]
  & \alpha^{\sharp} \in \mathfrak{X}(\S2) \ , \, \alpha \in  \Omega^1(\S2) ;
\end{align}
\end{subequations}
in coordinates, these correspond to lowering and raising indices, respectively.
Using flat and sharp, we can, for example, identify the Hodge-Helmholtz decomposition with the classical Helmholtz decomposition for vector fields and relate the exterior derivative to the classical differential operators of vector calculus.

\subsection{A Spectral Exterior Calculus for $\S2$}
\label{sec:forms:spectral}

In this section, we introduce a spectral exterior calculus for the sphere.
It can be found in various disguises in the literature, e.g. in the form of vector spherical harmonics.
An explicit formulation based on the de Rahm complex will provide us with a foundation for the construction of differential form wavelets in Sec.~\ref{sec:forms:psi}.

For the spectral exterior calculus we will systematically distinguish between exact, co-exact and harmonic forms, i.e. respect the Hodge-Helmholtz decomposition in Eq.~\ref{eq:hodge_helmholtz}.
We will do so by constructing distinct orthonormal bases, formed by the spectral differential form basis functions $y_{lm}^{\fdeg,\nu}(\omega)$, for the spaces
\begin{align}
  L_2(\Omega_{\nu}^{\fdeg}, \S2) = \spann_{\substack{-l \leq m \leq l \\ l \geq 0}}{ \big\{ y_{lm} ^{\fdeg,\nu}(\omega) \big\}} \, , \quad \nu \in \{ \dd , \delta, h \} .
\end{align}
We start with differential $0$-forms in $L_2(S^2) \cong L_2(\Omega_{\delta}^0,\S2) \bigoplus L_2(\Omega_{h}^0,\S2)$.
It is immediately apparent that in this case the $y_{lm} ^{\fdeg,\nu}(\omega)$ are given by
\begin{align}
  \label{eq:spectral_forms:0_forms}
    y_{lm}^{0,\delta}(\omega) =
  \left\{ \begin{array}{cc}
    0 & l = 0
    \\[3pt]
    y_{lm}(\omega) & \textrm{otherwise}
  \end{array} \right.
  \quad \quad \quad
  y_{lm}^{0,h}(\omega) =
  \left\{ \begin{array}{cc}
    y_{00}(\omega) & l = 0
    \\[3pt]
    0 & l \geq 0
  \end{array} \right.
\end{align}
where the $y_{lm}(\omega)$ are the usual scalar spherical harmonics.

The harmonic forms spanned by $y_{00}^{0,h}(\omega)$ are in the kernel of the exterior derivative $\dd$.
However, the image of $\Omega_{\delta}^0(S^2)$ under $\dd$ is precisely the space $\Omega_{\dd}^1(S^2)$ of exact $1$-forms.
By linearity, a basis for $\Omega_{\dd}^1(S^2)$ is thus obtained by taking the exterior derivative of the $y_{lm}^{0,\delta}(\omega)$.
We can define the $y_{lm}^{1,\dd}(\omega)$ thus by
\begin{subequations}
  \label{eq:spectral_forms:1_forms}
\begin{align}
  \dd y_{lm}^{0,\delta}(\omega) &= \sqrt{l(l+1)} \, y_{lm}^{1,\dd}(\omega) .
  \intertext{Through the Hodge dual, we can also obtain the basis forms $y_{lm}^{1,\delta}(\omega)$ for co-exact $1$-forms,}
  \star \, \dd y_{lm}^{0,\delta}(\omega) &= \sqrt{l(l+1)} \, y_{lm}^{1,\delta}(\omega) .
\end{align}
\end{subequations}
The $y_{lm}^{1,\nu}(\omega)$ defined by Eq.~\ref{eq:spectral_forms:1_forms} are covariant, normalized versions of the classical vector spherical harmonics $\vec{y}_{lm}^{1,\nu}(\omega) = (y_{lm}^{1,\nu}(\omega))^{\sharp}$.
It is well known, e.g.~\cite[Ch. 5.3]{Freeden1998}, that the $\vec{y}_{lm}^{1,\dd}(\omega)$ and $\vec{y}_{lm}^{1,\delta}(\omega)$ form together a complete basis for the space of $L_2$ vector fields on $\S2$.
By the musical isomorphisms, this implies that the $y_{lm}^{1,\dd}(\omega)$ and $y_{lm}^{1,\delta}(\omega)$ form an orthonormal basis for differential $1$-forms with $ y_{lm}^{1,\dd}(\omega)$ spanning the space of exact $1$-forms and $y_{lm}^{1,\delta}(\omega)$ those of co-exact ones.

We complete the spectral differential form basis functions by defining the $y_{lm}^{2,\nu}$ for $2$-forms as
\begin{subequations}
  \label{eq:spectral_forms:2_forms}
\begin{align}
  \dd \star \dd y_{lm}^0(\omega) &= l(l+1) \, y_{lm}^{2,\dd}(\omega) = \star \, y_{lm}^{0,\delta}(\omega)
  \\[4pt]
  \star y_{lm}^{0,h}(\omega) &= y_{lm}^{2,h}(\omega) ,
\end{align}
\end{subequations}
i.e. $y_{lm}^{2,\dd}(\omega) = y_{lm} \dd \theta \wedge \dd \phi$.
It follows from Eq.~\ref{eq:spectral_forms:0_forms} and the fact that the Hodge dual provides an isomorphism that the $y_{lm}^{2,\dd}(\omega)$ together with the $y_{lm}^{2,h}(\omega)$ form an orthonormal basis for $L_2(\Omega^2,\S2)$.
We summarize the foregoing construction in the following theorem.

\begin{theorem}
  The spectral differential form basis functions $y_{lm}^{\fdeg,\nu}(\omega)$ defined in Eqs.~\ref{eq:spectral_forms:0_forms}, \ref{eq:spectral_forms:1_forms},   \ref{eq:spectral_forms:2_forms} provide orthonormal bases for the spaces $L_2(\Omega_{\nu}^{r},S^2)$.
\end{theorem}

In contrast to most examples for numerical differential forms in the literature, e.g.~\cite{Desbrun2006,Arnold2018}, the $y_{lm}^{\fdeg,\nu}(\omega)$ above are forms in the sense of the continuous theory.
Thus, all operations available on them are also well defined for the $y_{lm}^{\fdeg,\nu}(\omega)$.
The question becomes, therefore, if the operations can be computed efficiently in numerical calculations.
We collect important results in this regard in the following theorem.

\begin{theorem}
  \label{prop:spectral_ec:properties}
  Let the $y_{lm}^{\fdeg,\nu}(\omega)$ be the spectral differential form basis function for $\S2$ defined in Eqs.~\ref{eq:spectral_forms:0_forms},~\ref{eq:spectral_forms:1_forms},~\ref{eq:spectral_forms:2_forms} and let $\vert \dd \vert = 0$,  $\vert \dd \vert = 1$,  $\vert h \vert = 0$.
   Then
   \vspace{0.8em}
  \begin{enumerate}[i.)]
    \setlength\itemsep{0.8em}

    \item Closure of exterior derivative: $\dd y_{lm}^{\fdeg,\delta} = \sqrt{l(l+1)} \, y_{lm}^{\fdeg+1,\dd}$, $\, \dd y_{lm}^{\fdeg,\dd} = 0$

    \item Closure of Hodge dual: $(-1)^{\fdeg \, \vert \nu \vert}$ $\star y_{lm}^{\fdeg,\nu} = y_{lm}^{2-\fdeg,\bar{\nu}}$

    \item Closure of co-differential: $\delta y_{lm}^{\fdeg,\dd} = \sqrt{l(l+1)} \, y_{lm}^{\fdeg-1,\delta}$

    \item Eigen-forms of Laplace-Beltrami operator: $\Delta y_{lm}^{\fdeg,\nu} = l(l+1) \, y_{lm}^{\fdeg,\nu}$

    \item Double-graded algebra structure: $y_{l_1 m_1}^{\fdeg_1,\nu} \wedge y_{l_2 m_2}^{\fdeg_2,\nu} \in \mathcal{H}_{l_1 + l_2}(\Omega^{\fdeg_1 + \fdeg_2})$

  \end{enumerate}
\end{theorem}
\vspace{0.5em}

\begin{proof}
  Properties i.) and ii.) are immediate consequences of the definition.
  Properties iii.) and iv.) then follow since $\delta = \star \dd \star$ and $\Delta = \dd \delta + \delta \dd$.
  The last property holds by the standard properties of spherical harmonics and the wedge product.
\end{proof}

From a practical point of view, Theorem~\ref{prop:spectral_ec:properties} implies, for example, that for $\alpha \in L_2(\Omega_{\delta}^{\fdeg},S^2)$ one has
\begin{align}
  \dd \alpha
  = \dd \left( \sum_{lm} \alpha_{lm} \, y_{lm}^{k,\delta} \right)
  = \sum_{lm} \alpha_{lm} \, \dd y_{lm}^{k,\delta}
  = \sum_{lm} \sqrt{l(l+1)} \, \alpha_{lm} \, y_{lm}^{k+1,\dd} .
\end{align}
Thus the basis functions coefficients $\alpha_{lm}$ are scaled by $\sqrt{l(l+1)}$ but otherwise invariant under the exterior derivative and the different degree of $\dd \alpha$ is realized by performing the reconstruction with the $y_{lm}^{k+1,\dd}(\omega)$ form basis functions of degree $r+1$.
An analogous observation holds for the Hodge dual.
In other words, Theorem~\ref{prop:spectral_ec:properties} shows that in the spectral exterior calculus both the exterior derivative, $\dd$, and the Hodge dual, $\star$, are diagonal operators.
Numerically, they can hence be computed very efficiently.
Furthermore, they preserve the bandlimit for a bandlimited function.
Theorem~\ref{prop:spectral_ec:properties} also shows why it is convenient to use separate bases for the exact and co-exact form: the domain, range and kernel of the exterior derivative are cleanly separated.
Property iv.) in the theorem reflects the doubly graded structure of spectral differential forms with one grading in the degree $\fdeg$ of the forms and a second one in the harmonic degree $l$.
The product there can also be expressed precisely using Clebsch-Gordon coeffcients, cf.~\cite[Sec. 3]{Gia2019}, but since we will not need it in the following we leave this to future work.
An alternative way to evaluate the wedge product is the transform method, i.e. the evaluation of the product in the spatial domain.


\begin{remark}[Stokes' theorem]
  \label{remark:stokes_theorem:spectral}
  An important result in the exterior calculus is Stokes' theorem,
  \begin{subequations}
    \label{eq:stokes:spectral}
  \begin{align}
    \label{eq:stokes:spectral:1}
    \int_{\partial U} \alpha = \int_{U} \dd \alpha
  \end{align}
  for $\alpha \in \Omega^{\fdeg}(\S2)$ and $U \subseteq \S2$.
  Exploiting the Hodge-Helmholtz decomposition, expanding both sides in the respective spectral differential form basis functions and using linearity we can write Eq.~\ref{eq:stokes:spectral:1} as
  \begin{align}
    \label{eq:stokes:spectral:4}
     \sum_{lm} \alpha_{lm} \int_{\partial U} y_{lm}^{\fdeg,\delta}
    &=  \sum_{lm} \alpha_{lm} \sqrt{l(l+1)} \int_{U} \, y_{lm}^{\fdeg+1,\dd} .
  \end{align}
  \end{subequations}
  The integrals above have, in general, to be solved numerically using quadrature rules for the sphere~\cite{Hesse2010a}.
  Due to the global support of the spherical harmonics, the locality of $U$ can thereby, however, not be exploited and it is not obvious that either of the integrals is easier to evaluate numerically.
  We will return to this observation in the next section when we discuss the differential form wavelets.

  Although Eq.~\ref{eq:stokes:spectral:4} is of limited practical relevance, in the special case when $U$ is a spherical cap it provides insight into how geometric and functional analytic properties interact in the spectral exterior calculus.
  Let $\mathcal{C}^{\gamma}$ be a spherical cap of opening angle $\gamma$ and with boundary $\partial \mathcal{C}^{\gamma}$ and, without loss of generality, assume that it is centered at the North Pole.
  Also let $\alpha$ be a $1$-form.
  To compute Eq.~\ref{eq:stokes:spectral:4} we write the integrals using the characteristic functions $\chi_{_{\mathcal{C}^{\gamma}}}$ and $\chi_{_{\partial \mathcal{C}^{\gamma}}}$, i.e. as
  \begin{align}
    \label{eq:stokes:spectral:cap:0}
         \sum_{lm} \alpha_{lm} \int_{\S2} \chi_{_{\partial \mathcal{C}^{\gamma}}} \cdot y_{lm}^{1,\delta}
    &=  \sum_{lm} \alpha_{lm} \sqrt{l(l+1)} \int_{\S2} \chi_{_{\mathcal{C}^{\gamma}}} \, y_{lm}^{2,\dd}  .
  \end{align}
  By the orthonormality of the spectral differential form basis functions, it is convenient to compute the integrals using the basis representations of the characteristic functions.
  Their closed form representations are given by
  \begin{subequations}
  \label{eq:stokes:spectral:cap}
  \begin{align}
    \label{eq:stokes:spectral:cap:1}
    \mathcal{C}_{lm}^{\gamma} &= C_{lm} \, \frac{P_{l-1}(\cos{\gamma}) - P_{l+1}(\cos{\gamma})}{2l+1} \, \delta_{m0}
    \\[4pt]
    \label{eq:stokes:spectral:cap:2}
    \partial \mathcal{C}_{lm}^{\gamma} &= C_{lm} \, P_l(\cos{\gamma}) \, \delta_{m 0} .
  \end{align}
  The coefficients $\mathcal{C}_{lm}^{\gamma}$ decay as $\mathcal{O}(1/l)$ in $l$ while the $\partial \mathcal{C}_{lm}^{\gamma}$ do as $\mathcal{O}(1)$.
  This reflects that the boundary $\partial \mathcal{C}^{\gamma}$ is (functional analytically) more singular than the domain $\mathcal{C}^{\gamma}$.
  For the integral over the boundary on the left hand side of Eq.~\ref{eq:stokes:spectral:cap:0} only the component of $y_{lm}^{1,\delta}$ tangential to it, given by $( y_{lm}^{1,\delta} )_2$, is required.\footnote{More correctly, the pullback along the inclusion map $i : \partial \mathcal{C}^{\alpha} \to \S2$ should be used here.}
  Using the recurrence relationship for the Legendre polynomials it can be written as
  \begin{align}
    \label{eq:stokes:spectral:cap:3}
    ( y_{l0}^{1,\delta} )_2 
    = \frac{\partial}{\partial \theta} \frac{y_{l0}}{\sqrt{l(l+1)}} = C_{l0} \, \frac{\sqrt{l(l+1)}}{2l+1} \big(  P_{l-1}(\cos{\alpha}) - P_{l+1}(\cos{\alpha}) \big)
  \end{align}
  where we only consider $m=0$ since, by the $\phi$-symmetry of the spherical cap, cf. Eq.~\ref{eq:stokes:spectral:cap}, only these will yield a nontrivial contribution to the integrals in Eq.~\ref{eq:stokes:spectral:cap:0}.
  Inserting Eqs.~\ref{eq:stokes:spectral:cap} into Eq.~\ref{eq:stokes:spectral:4} and using the orthonormality of the $y_{lm}^{\fdeg,\nu}$ we obtain for the integrals in the equation,
  \begin{align}
    \int_{\S2} \! \! \chi_{_{\partial \mathcal{C}^{\gamma}}} \cdot y_{lm}^{1,\delta}
    &= \big( C_{l0} \,  P_l(\cos{\gamma}) \big) \Bigg( \! C_{l0} \frac{\sqrt{l(l+1)}}{2l+1} \big(  P_{l-1}(\cos{\gamma}) - P_{l+1}(\cos{\gamma}) \big) \! \Bigg)
    \\[5pt]
    \int_{\S2} \chi_{_{\mathcal{C}^{\gamma}}} \, y_{lm}^{2,\dd}
    &= \Bigg( C_{l0} \, \frac{P_{l-1}(\cos{\gamma}) - P_{l+1}(\cos{\gamma})}{2l+1} \Bigg) \, \big(  C_{l0} \, P_l(\cos{\gamma} ) \big)
  \end{align}
  \end{subequations}
  The equality in Eq.~\ref{eq:stokes:spectral:4} thus, indeed, holds.
  The last equations, furthermore, show that the regularity that is lost by going from $\alpha$ to $\dd \alpha$, in the form of the gain factor $\sqrt{l(l+1)}$, is compensated by the higher regularity of $\mathcal{C}^{\gamma}$ compared to $\partial \mathcal{C}^{\gamma}$, expressed in the different decay rates in Eq.~\ref{eq:stokes:spectral:cap:1} and Eq.~\ref{eq:stokes:spectral:cap:2}.
\end{remark}

\begin{remark}
  As remarked before, in the spectral exterior calculus $\dd$ and $\star$ are exactly satisfied, in the sense of the continuous theory, and both are realized by diagonal operators.
  We are not aware of another numerically practical discretization of exterior calculus with these properties.
  A diagonal Hodge dual matrix also appears in other approaches, e.g. Discrete Exterior Calculus~\cite{Desbrun2006}.
  However, there it only provides an approximation, akin to mass lumping in classical finite elements, to increase computational efficiency.
  Similar to other higher-order discretizations~\cite{Rapetti2009,Hiptmair2001,Rufat2014,Gross2018a}, the spectral exterior calculus combines structure preservation with spectral accuracy, i.e. optimal convergence rates for smooth problems.
  For us, the accuracy comes at the usual price for spectral methods, namely global support.
  This makes the approach ill suited for problems with strongly varying regularity or on subdomains of $\S2$ and provides a principal motivation for the local spectral exterior calculus $\Psiec$.
\end{remark}
%


\begin{remark}[Relationship to ``Spectral Exterior Calculus'' by Berry and Giannakis~\cite{Berry2018}]
  \label{remark:berry_giannakis}
  Recently, Berry and Giannakis introduced a spectral exterior calculus for manifold learning problems.
  The work is also based on eigenfunctions of the Laplace-Beltrami operator.
  However, these authors assume that only data points $x_i$ sampled from a manifold $\mathcal{M}$ are given and these are to be used to recover properties of $\mathcal{M}$.
  Our spectral exterior calculus, in contrast, aims at the solution of partial differential equations on a fixed manifold, in our case $\S2$, and it relies on an explicit representation of the Laplacian eigenfunctions there.
  Our construction could be generalized to other embedded $2$-manifolds but it becomes useful only when the eigenfunctions are known (at least numerically, cf.~\cite{Vallet2008}).
\end{remark}

\subsection{Homogeneous Sobolev Spaces for $\Omega_{\nu}^{\fdeg}(\S2)$}
\label{sec:forms:sobolev}

The functional analytical setting for $\Psiec$ will be the homogeneous Hilbert-Sobolev spaces $\dot{H}^s(\Omega_{\nu}^{\fdeg},\S2)$.
These spaces will hence be introduced next.

For scalar functions, i.e. differential forms of degree $0$, the homogeneous Sobolev space can be defined as
\begin{align}
  \label{eq:def:homogeneous_sobolev:freq}
  \dot{H}^s(\S2) = \Big\{ f : \S2 \to \R \, \Big\vert \, \Vert f \Vert_{2,s} = \sum_{l=0}^{\infty} \sum_{m=-l}^l (l(l+1))^{s} \, \vert f_{lm} \vert^2 < \infty , \,  f_{00} = 0 \Big\}
\end{align}
where $s \in \R$ is the regularity order and the $f_{lm}$ are the $L_2$ spherical harmonics coefficients of $f$.
In contrast to classical Sobolev spaces $H^s(\S2)$ on the sphere, where the weight is $(1+l)^{2s}$ (or a norm-equivalent choice, see e.g.~\cite{Brauchart2014,LeGia2010}), for homogeneous ones the weight function vanishes for $l=0$.
This implies that for the spaces to be Hilbert either the auxiliary condition $f_{00} = 0$ is required or one has to work with appropriate co-sets~\cite[Ch.~2]{Galdi2011}.

An alternative to these choices is to respect the structure of the Hodge-Helmholtz decomposition and define the homogeneous Sobolev space for $0$-forms only on the space of co-exact ones where, by construction, $f_{00} = 0$.
The definition can then, furthermore, be carried over to $1$- and $2$-forms using the expansions in the spectral differential form basis functions $y_{lm}^{\fdeg,\nu}(\omega)$.
For $\nu = \{ \dd , \delta \}$, we therefore have
\begin{align}
  \label{eq:def:sobolev:space:forms}
  \dot{H}^s(\Omega_{\nu}^{\fdeg},\S2)
  &= \Big\{ \alpha \in \Omega_{\nu}^{\fdeg}(\S2) \  \Big\vert \ \Vert \alpha \Vert_{2,s} = \sum_{l=0}^{\infty} \sum_{m=-l}^l (l(l+1))^{s} \, \vert \alpha_{lm}^{\fdeg,\nu} \vert^2 < \infty  \Big\}
\end{align}
where the $\alpha_{lm}^{\fdeg,\nu}$ are the spectral coefficients of the $\fdeg$-form $\alpha$, i.e. $\alpha_{lm}^{\fdeg,\nu} = \llangle \alpha , y_{lm}^{\fdeg,\nu} \rrangle$.

Eq.~\ref{eq:def:sobolev:space:forms} corresponds to the Sobolev spaces of the second kind for differential forms discussed by Dodziuk~\cite{Dodziuk1981}.
These are defined in the spatial domain using the Laplace-Beltrami operator $\Delta$.
Indeed, using the symbol $\hat{\Delta}$ of $\Delta$ in the spherical harmonics domain, $\hat{\Delta} = l(l+1)$, it is not difficult to see that for $s = 1$ the definition in Eq.~\ref{eq:def:sobolev:space:forms} is equivalent to
\begin{subequations}
\begin{align}
  \alpha \in \Omega_{\delta}^{\fdeg} \ &: \ \big\Vert \alpha \big\Vert_{\dot{H}^1(\Omega_{\delta}^r)}^2 = \llangle \dd \alpha , \dd \alpha \rrangle = \llangle \alpha , \Delta \alpha \rrangle = \int_{\S2} \alpha \wedge \star \Delta \alpha
  \\[4pt]
  \beta \in \Omega_{\dd}^{\fdeg} \ &: \ \big\Vert \beta \big\Vert_{\dot{H}^1(\Omega_{\dd}^r)}^2 = \llangle \delta \beta , \delta \beta \rrangle = \llangle \Delta \beta , \beta \rrangle = \int_{\S2} \Delta \beta \wedge \star \beta .
\end{align}
\end{subequations}

The dual space of $\dot{H}^{1}(\Omega_{\nu}^{\fdeg},\S2)$ is the space of distributions $\dot{H}^{-1}(\Omega_{\nu}^{\fdeg},\S2)$.
For $\fdeg=0$ and $\fdeg=2$ the usual scalar theory applies.
For $r=1$ the duality pairing is defined using the Hodge-Helmholtz decomposition of $\gamma \in \dot{H}^{-1}(\Omega^{1},S^2)$ given by $\gamma = \gamma_{\dd} + \gamma_{\delta} = \dd \alpha + \delta \beta$ with $\alpha \in \dot{H}^{0}(\Omega^{0},S^2)$ and $\beta \in \dot{H}^{0}(\Omega^{2},S^2)$.
Thus, for $\zeta_{\dd} \in \dot{H}^{1}(\Omega_{\dd}^{1},S^2)$ and $\zeta_{\delta} \in \dot{H}^{1}(\Omega_{\delta}^{1},S^2)$ the following non-degenerate pairings are well defined
\begin{subequations}
\label{eq:Hm1:duality_pairing}
\begin{align}
  \big\langle \! \big\langle \gamma_{\dd} , \zeta_{\dd} \big\rangle \! \big\rangle
  &= \big\langle \! \big\langle \dd \alpha , \zeta_{\dd} \big\rangle \! \big\rangle
  = \big\langle \! \big\langle \alpha , \delta \zeta_{\dd} \big\rangle \! \big\rangle
  = \int_{\S2} \alpha \wedge \star \delta \zeta_{\dd}
  \\[4pt]
    \big\langle \! \big\langle \gamma_{\delta} , \zeta_{\delta} \big\rangle \! \big\rangle
  &= \big\langle \! \big\langle \delta \beta , \zeta_{\delta} \big\rangle \! \big\rangle
  = \big\langle \! \big\langle \beta , \dd \zeta_{\delta} \big\rangle \! \big\rangle
  = \int_{\S2} \beta \wedge \star \dd \zeta_{\delta}
\end{align}
\end{subequations}
where it is easy to check that the integrands on the right hand side are, indeed, volume forms.
De Rahm~\cite{deRahm1984} introduced the term `current' to denominate differential forms whose coordinate function are distributions in the sense of Schwartz.
The space $\dot{H}^{-1}(\Omega^\fdeg,\S2)$ can also be defined in its own right, see e.g.~\cite{Strichartz1983}, but for us the duality in Eq.~\ref{eq:Hm1:duality_pairing} suffices.

\begin{remark}[Connection to $H(\mathrm{curl},\S2)$ and $H(\mathrm{div},\S2)$]
  \label{remark:Hcurl}
  The spaces $H(\mathrm{curl})$ and $H(\mathrm{div})$ of $L_2$-vector fields whose curl respectively divergence is also in $L_2$ provide the standard setting for finite element-type discretizations of exterior calculus~\cite{Nedelec1980,Girault1986,Hiptmair2002,Arnold2018}.
  On $\S2$, by the (Hodge-)-Helmholtz decomposition an arbitrary vector field $\vec{u} \in \mathfrak{X}(\S2)$ is given by
  \begin{align}
    \vec{u} = \sum_{l=1}^{\infty} \sum_{m=-l}^l u_{lm}^{\dd} \, \vec{y}_{lm}^{1,\dd} + \sum_{l=1}^{\infty} \sum_{m=-l}^l u_{lm}^{\delta} \, \vec{y}_{lm}^{1,\delta}
  \end{align}
  where the $\vec{y}_{lm}^{1,\dd}$, $\vec{y}_{lm}^{1,\delta}$ are the orthonormal vector spherical harmonics, i.e. the contravariant versions of the spectral differential forms $y_{lm}^{1,\dd}$, $y_{lm}^{1,\delta}$ of Sec.~\ref{sec:forms:spectral}.
  For the curl we have
  \begin{align}
    \mathrm{curl}(\vec{u})
    = ( \dd u^{\flat} )^{\sharp}
    = \sum_{l,m} u_{lm}^{1,\delta} \, \sqrt{l(l+1)} \, y_{lm} .
  \end{align}
  The space $H(\mathrm{curl},\S2)$ can thus also be characterized as
  \begin{align}
    H(\mathrm{curl},\S2)
    &= \Big\{ \vec{u} \in \mathfrak(\S2) \ \Big\vert \ \big\Vert u_{lm}^{1,\dd} \big\Vert_{\ell_2} < \infty, \, \big\Vert \sqrt{l(l+1)} \, u_{lm}^{1,\delta} \big\Vert_{\ell_2} < \infty  \Big\}
  \end{align}
  where $\Vert u_{lm}^{1,\delta} \Vert_{\ell_2} < \infty$ is automatically satisfied.
  Equivalently, we have
  \begin{align}
    \label{eq:H_curl:forms}
    H(\mathrm{curl},\S2)^{\flat} = \dot{H}^0( \Omega_{\dd}^1 , \S2 ) \oplus \dot{H}^1( \Omega_{\delta}^1 , \S2 )
  \end{align}
  with the flat on the left hand side being understood element-wise.
  Analogously,
  \begin{align}
    \label{eq:H_div:forms}
    H(\mathrm{div},\S2)^{\flat} = \dot{H}^1( \Omega_{\dd}^1 , \S2 ) \oplus \dot{H}^0( \Omega_{\delta}^1 , \S2 )  .
  \end{align}
  For the application we have in mind, it is natural and convenient to keep exact and co-exact parts of $1$-forms (or the associated vector fields) separate.
  In the literature, Hiptmair, Li, and Zou~\cite{Hiptmair2012} similarly define a separate space for closed differential forms, which they denote as $H(\dd 0 , \R^d,  \Omega^l)$.
  Compared to Eq.~\ref{eq:H_curl:forms} and Eq.~\ref{eq:H_div:forms}, we work with one degree of regularity less to obtain closure under the Hodge dual, see Fig.~\ref{fig:psiec_s2}.

  Finite element exterior calculus~\cite{Arnold2018} uses Hilbert complexes, a concept first introduced for Hodge theory as the functional analytic setting for exterior calculus~\cite{Bruning1992}.
  A Hilbert complex is a sequence of Hilbert spaces $W^{\fdeg}$ with a densely defined, closed linear operator $\dd^{\fdeg} : W^{\fdeg} \to W^{\fdeg+1}$ that maps its domain into the kernel of $\dd^{\fdeg+1}$.
  The complex is closed if $\dd^{\fdeg}$ has closed range.
  This provides a general setting for the Hodge-Helmholtz decomposition.
  The connection between our formulation using homogeneous Sobolev spaces and the framework of Hilbert complexes is left to future work.


\end{remark}

\subsection{Wavelet Differential Forms for $\S2$}
\label{sec:forms:psi}

In the following, we will introduce a consistent set of frames for differential $\fdeg$-forms that satisfies important properties of the exterior calculus.
In analogy to existing discretizations, such as DEC~\cite{Hirani2003,Desbrun2006} and FEEC~\cite{Arnold2006}, we refer to it as $\Psiec$.

The principle idea of $\Psiec$ is to use the discrete window functions $\kappa_{l}^{j}$ of the scalar wavelets in Sec.~\ref{sec:wavelets} with the spectral differential form basis functions $y_{lm}^{\fdeg,\nu}(\omega)$, analogous to Eq.~\ref{eq:mother_wavelet}.
By linearity of the Hilbert space structure as well as the exterior derivative, this leads to differential form wavelets that respect essential properties of the exterior calculus and that are well localized in space and frequency.
Although one could work with tight frames for $0$, $1$-, and $2$-forms, closure under the Hodge dual is obtained by working with Stevenson frames for $0$-forms and $2$-forms, which result through an $l$-dependent, Sobolev-type weight.
For the Stevenson frames, primary and dual frame functions $\smash{\psi_{jk}^{\fdeg,\delta}(\omega)}$ and $\smash{\tilde{\psi}_{jk}^{\fdeg,\delta}(\omega)}$ are not identified (using the Riesz representation theorem~\cite{Stevenson2003,Balazs2019}) but they form frames for the dual spaces $\dot{H}^{-\fdeg+1}(\Omega^{\fdeg})$ and $\dot{H}^{\fdeg-1}(\Omega^{\fdeg})$, respectively.
Since the $\smash{\psi_{jk}^{\fdeg,\nu}(\omega)}$ and $\smash{\tilde{\psi}_{jk}^{\fdeg,\nu}(\omega)}$ have analytic expressions and differ only by the $l$-dependent weight, the practical difference to a tight frame is limited albeit the weighting affects the localization when the frame functions are considered individually.
Similar to~\cite{Balazs2019}, we will also distinguish norms despite them being equivalent in the finite dimensional spaces spanned by the differential $\fdeg$-form wavelets.
This is conceptually and numerically advantageous  in our case.

We begin with the definition of differential form wavelets.

\begin{definition}
  \label{def:psiec:psi_forms}
  Let the $L_j$-bandlimited scalar wavelets $\psi_{jk}(\omega)$, with $L_j < \infty$, form a tight wavelet frame for $L_2(\S2)$ and let $\kappa_{lm}^{jk} = \sqrt{ 4\pi / 2l+1} \, \sqrt{w_{jk}} \, \kappa_l^j \, y_{lm}(\lambda_{jk})$ be the associated spherical harmonics coefficients.
  Furthermore, let $a_l = \sqrt{l(l+1)}$.
  The \emph{spherical wavelet differential $\fdeg$-forms $\psi_{jk}^{\fdeg,\nu}(\omega)$ and their duals $\tilde{\psi}_{jk}^{\fdeg,\nu}(\omega)$} are then

  \noindent
  \begin{tabular}{lc}
   $\Omega^0(S^2) \quad$ &
     $\!\begin{aligned}
    \psi_{jk}^{0,h}(\omega) = \sum_{l=0}^{L_j} & \sum_{m=-l}^l \kappa_{lm}^{jk} \, y_{lm}^{0,h}(\omega) =  \kappa_{00}^{jk} \, y_{00}^{0,h}(\omega)
    \\[3pt]
    \psi_{jk}^{0,\delta}(\omega) = \sum_{l=0}^{L_j} \! \sum_{m=-l}^l \! \! a_l^{-1} \, \kappa_{lm}^{jk} & \, y_{lm}^{0,\delta}(\omega)
    \quad \quad
    \tilde{\psi}_{jk}^{0,\delta}(\omega) = \sum_{l=0}^{L_j} \! \sum_{m=-l}^l \! \! a_l \, \kappa_{lm}^{jk} \, y_{lm}^{0,\delta}(\omega)
    \end{aligned}$
    \\[-5pt]
    \\ \hdashline
    \\[-5pt]
    $\Omega^1(S^2) \quad$ &
    $\!\begin{aligned}
      \psi_{jk}^{1,\dd}(\omega) &= \sum_{l=0}^{L_j} \sum_{m=-l}^l \kappa_{lm}^{jk} \, y_{lm}^{1,\dd}(\omega)
      \\[3pt]
      \psi_{jk}^{1,\delta}(\omega) &= \sum_{l=0}^{L_j} \sum_{m=-l}^l \kappa_{lm}^{jk} \, y_{lm}^{1,\delta}(\omega)
    \end{aligned}$
    \\[-5pt]
    \\ \hdashline
    \\[-5pt]
    $\Omega^2(S^2) \quad$ &
    $\!\begin{aligned}
    \psi_{jk}^{2,\dd}(\omega) = \sum_{l=0}^{L_j} \! \sum_{m=-l}^l \! \! a_l \, \kappa_{lm}^{jk} & \, y_{lm}^{2,\dd}(\omega)
    \quad \ \
    \tilde{\psi}_{jk}^{2,\dd}(\omega) = \sum_{l=0}^{L_j} \! \sum_{m=-l}^l \! \! a_l^{-1} \, \kappa_{lm}^{jk} \, y_{lm}^{2,\dd}(\omega)
    \\[3pt]
    \psi_{jk}^{2,h}(\omega) = \sum_{l=0}^{L_j} & \sum_{m=-l}^l \kappa_{lm}^{jk} \, y_{lm}^{2,h}(\omega) = \kappa_{00}^{jk} \, y_{00}^{2,h}(\omega)
    \end{aligned}$
  \end{tabular}
\end{definition}
\vspace{5pt}

The $1$-form wavelets $\psi_{jk}^{1,\nu}(\omega)$ as well as the harmonic ones $\psi_{jk}^{\fdeg,h}(\omega)$ are self-dual and hence the dual frame functions are not explicitly listed.

The harmonic forms in $\Omega_h^0(S^2)$ and $\Omega_h^2(S^2)$ are the constants.
It is hence advantageous to select the scaling function windows as $\kappa_{lm}^{-1,0} = \delta_{l0}$.
The scaling functions then represent the harmonic forms on $\S2$ and the wavelets, which have spectral support with $l \geq 1$, cover the exact and co-exact forms.
Unless mentioned otherwise, we will assume this in the following.

\begin{figure}
  \includegraphics[width=\textwidth]{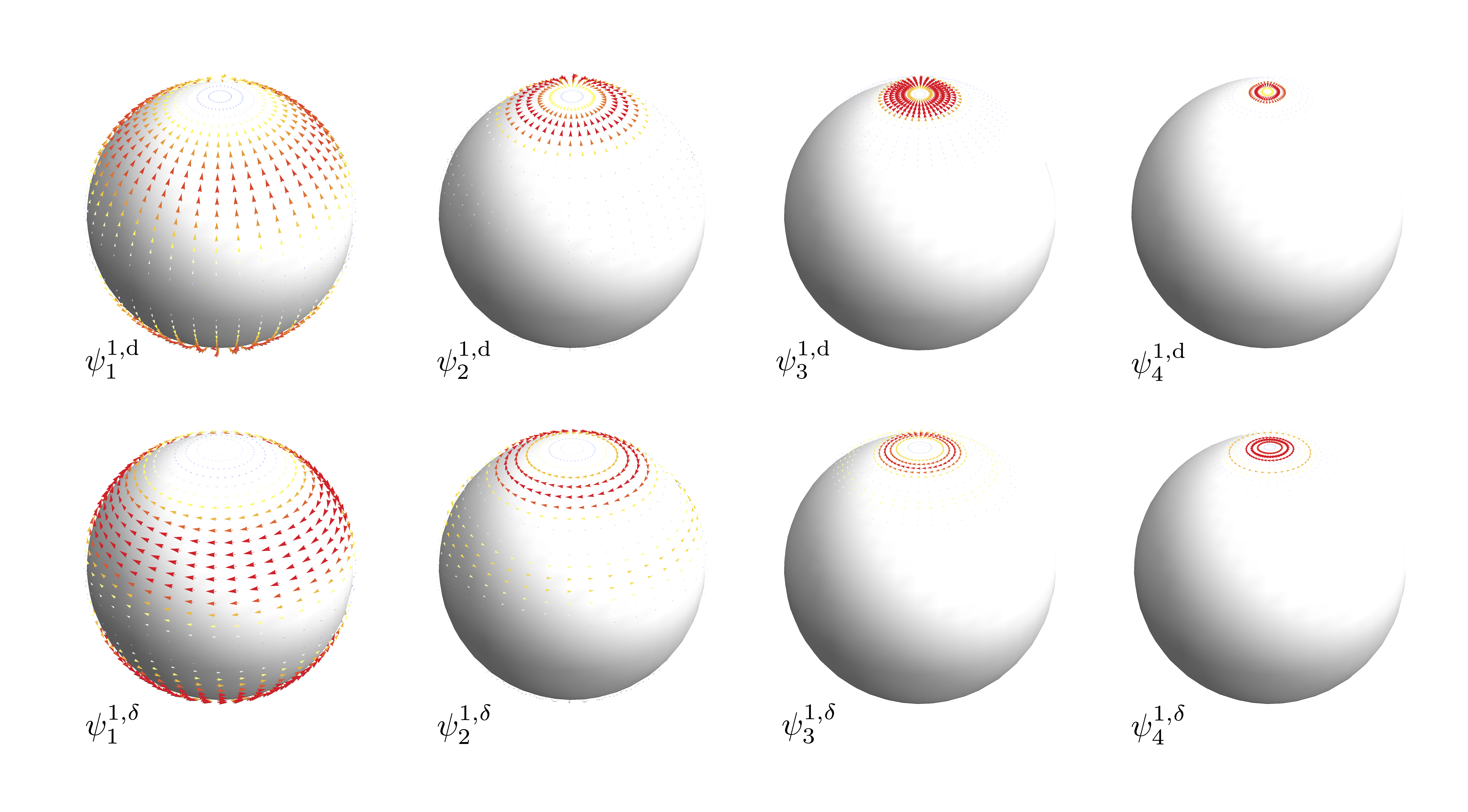}
  \caption{Visualization of exact and co-exact $1$-form mother wavelets, $\psi_j^{1,\dd}(\omega)$ (top) and $\psi_j^{1,\delta}(\omega)$ (bottom), respectively, for $j=1,2,3,4$. The exact ones correspond to curl free vector fields and they are close to a perfect, localized sink at the North pole. Correspondingly, the co-exact wavelets are isomorphic to divergence free vector fields and they are close to a perfect, localized vortex around the pole. }
  \label{fig:sh1_grid_up}
\end{figure}

The next theorem establishes that for exact and co-exact forms, i.e. $\nu \in \{ \dd , \delta \}$ the differential $\fdeg$-form wavelets form (Stevenson) frames for the homogeneous Sobolev spaces $\dot{H}^1(\Omega_{\delta}^0,\S2)$, $\dot{H}^0(\Omega_{\nu}^1,\S2) \cong L_2(\Omega_{\nu}^1 , \S2)$, and $\dot{H}^{-1}(\Omega_{\dd}^2 ,\S2)$.

\begin{theorem}
  \label{thm:diff_form_wavelets:frame_property}
  For $\nu \in \{ \dd , \delta \}$, the differential $\fdeg$-form wavelets $\psi_{jk}^{\fdeg,\nu}(\omega)$ in Def.~\ref{def:psiec:psi_forms} provide (Stevenson) frames for the spaces $\dot{H}^{-\fdeg+1}( \Omega_{\nu}^{\fdeg} , \S2 )$ with the $\tilde{\psi}_{jk}^{\fdeg,\nu}(\omega)$ being the canonical dual frames in $\dot{H}^{\fdeg-1}( \Omega_{\nu}^{\fdeg} , \S2 )$.
\end{theorem}

\begin{proof}
  We compute the case $\fdeg=0$, $\nu=\delta$; the other ones follow by analogous calculations.
  We first verify that $\psi_{jk}^{0,\delta}(\omega) \in \dot{H}^{1}(\Omega_{\delta}^0,\S2)$.
  Using the spatial definition of the homogeneous $\dot{H}^1$ Sobolev inner product in Eq.~\ref{eq:def:sobolev:space:forms} we have
  \begin{subequations}
  \begin{align}
    \Big\Vert \psi_{jk}^{0,\delta} \Big\Vert_{\dot{H}^{1}}
    = \Big\langle \dd \psi_{jk}^{0,\delta} , \dd \psi_{jk}^{0,\delta} \Big\rangle_{\dot{H}^{0}}
    = \int_{\S2} \psi_{jk}^{0,\delta} \, \Delta \overline{\psi_{jk}^{0,\delta}} \, d\omega
  \end{align}
  where the overbar denotes complex conjugation.
  With Def.~\ref{def:psiec:psi_forms} and that the scalar spherical harmonics are eigenfunctions of the Laplacian we obtain
  \begin{align}
    \Big\Vert \psi_{jk}^{0,\delta} \Big\Vert_{\dot{H}^{1}}
    &=
   \sum_{l,m} \sum_{l',m'}  \frac{\kappa_{lm}^{jk}}{\sqrt{l(l+1)}} \frac{\overline{\kappa_{l'm'}^{jk}} \, l'(l'+1)}{\sqrt{l'(l'+1)}} \int_{\S2}  y_{lm}^{0,\delta}(\omega) \, \overline{y_{l'm'}^{0,\delta}(\omega)} \, d\omega
    = \sum_{l,m} \kappa_{lm}^{jk} \, \overline{\kappa_{lm}^{jk}}
    \nonumber
  \end{align}
  which is finite since it is true for the scalar wavelets and the filter taps $\kappa_{lm}^{jk}$ are in $\ell_2$.
  For the frame property, we start from the representation in the spectral differential form wavelets.
  Using Lemma~\ref{prop:wavelets:calderon} it can be written as
  \begin{align}
     f(\omega) &= \sum_{l,m} \sum_{l',m'} \frac{\sqrt{l(l+1)}}{\sqrt{l'(l'+1)}} \, f_{lm} \, y_{l'm'}^{0,\delta}(\omega) \sum_{j=-1}^{\infty} \sum_{k \in \mathcal{K}_j} \overline{\kappa_{lm}^{jk}} \, \kappa_{l'm'}^{jk} .
      \\[5pt]
      &= \sum_{j=-1}^{\infty} \sum_{k \in \mathcal{K}_j} \Bigg( \sum_{l,m} f_{lm} \, \Big( \sqrt{l(l+1)}\, \overline{\kappa_{lm}^{jk}} \Big) \Bigg) \frac{\kappa_{l'm'}^{jk}}{\sqrt{l'(l'+1)}} \, y_{l'm'}^{0,\delta}(\omega)
  \end{align}
  Using the definition of the primary and dual wavelets this equals
  \begin{align}
    \label{eq:}
    f(\omega) = \sum_{j=-1}^{\infty} \sum_{k \in \mathcal{K}_j} \big\langle f(\eta) , \tilde{\psi}_{jk}^{0,\delta}(\eta) \big\rangle \, \psi_{jk}^{0,\delta}(\omega) .
  \end{align}
  \end{subequations}
  The result follows now from Lemma~\ref{prop:wavelets:calderon}.
\end{proof}

The foregoing result establishes representability of differential $\fdeg$-forms with the wavelets $\psi_{jk}^{\fdeg,\nu}(\omega)$.
The next theorem shows that these provide a local spectral exterior calculus in that important operations of the exterior calculus are naturally defined.
It is the analogue of Theorem~\ref{prop:spectral_ec:properties} for spectral differential forms.

\begin{theorem}
  \label{thm:psiec}
  The spherical wavelet differential $\fdeg$-forms $\psi_{jk}^{\fdeg,\nu}(\omega)$ of Def.~\ref{def:psiec:psi_forms} satisfy:
   \vspace{0.8em}
  \begin{enumerate}[i.)]
    \setlength\itemsep{0.8em}

    \item Closure of exterior derivative: $\dd \psi_{jk}^{\fdeg,\delta} = \psi_{jk}^{\fdeg+1,\dd}$, $\dd \psi_{jk}^{\fdeg,\dd} = 0$

    \item Closure of Hodge dual: $(-1)^{\fdeg \, \vert \nu \vert} \star \psi_{jk}^{\fdeg,\nu} = \tilde{\psi}_{jk}^{n-\fdeg,\bar{\nu}}$

    \item Closure of co-differential: $\delta \psi_{jk}^{1,\dd} = \tilde{\psi}_{jk}^{0,\delta}$, $\delta \tilde{\psi}_{jk}^{2,\dd} = \psi_{jk}^{1,\delta}$,

    \item Closure of Laplacian for $0$-, $2$-forms: $\Delta \psi_{jk}^{0,\delta} = \tilde{\psi}_{jk}^{0,\delta}$, $\Delta \tilde{\psi}_{jk}^{2,\dd} = \psi_{jk}^{2,\dd}$,

    \item Double-graded algebra structure: $\psi_{j_1 k_1}^{\fdeg_1,\nu} \wedge \psi_{j_2 k_2}^{\fdeg_2,\nu} \in \mathcal{H}_{2^{j_1+j_2+2}}(\Omega_{\nu}^{\fdeg_1 + \fdeg_2},\S2)$

  \end{enumerate}
\end{theorem}
\vspace{0.2em}

\begin{proof}
  Property i.) and ii.) are an immediate consequence of the definition of wavelet differential forms.
  Property iii.) then follows since the co-differential is given by $\delta = \star \dd \star$ and this in turn implies Property iv.) by $\Delta = \delta \dd + \dd \delta$.
  Finally,  the last property holds by the analogous property for spherical harmonics and the linearity of the wedge product.
\end{proof}

Theorem~\ref{thm:psiec} is summarized in Fig.~\ref{fig:psiec_s2}.
Implicit in Property i.) is that the maximum level $J$ that is used for a representation in practical numerical calculations is invariant under the exterior derivative, which implies that a finite representation remains finite and of the same dimension.
This property plays an important role for numerical calculations.
Property iii.) is the reason that we work with the dual forms $\tilde{\psi}_{jk}^{\fdeg,\bar{\nu}}$ since these provide us with closure for the Hodge dual, i.e. it can be represented without the need for a projection.
The only essential operation where $\Psiec$ is not closed is the Laplacian for $1$-forms and the wedge product.
For the latter, Property v.) provides, however, enough control for the fast transform method to be applicable, as will be demonstrated in Sec.~\ref{sec:shallow}.

To make the above construction more concrete from an application point of view, we consider two examples.

\begin{example}[Poisson's equation]
  \label{ex:poisson_equation}
  We consider Poisson's equation $\Delta \omega = \beta$ for $\omega \in \dot{H}^1(\Omega_{\dd}^2,S^2)$ and $\beta \in \dot{H}^{-1}(\Omega_{\dd}^2,S^2)$.
  With the representation of $\omega$ in the dual differential form wavelets $\tilde{\psi}_{s}^{2,\dd}$ we have
  \begin{subequations}
  \begin{align}
    \beta
    = \Delta \sum_{j=0}^{J} \sum_{k=1}^{\vert \Lambda_j \vert} \omega_{jk} \, \tilde{\psi}_{jk}^{2,\dd}
    = \sum_{j=0}^{J} \sum_{k=1}^{\vert \Lambda_j \vert} \omega_{jk} \, \psi_{jk}^{2,\dd} .
  \end{align}
  where we used that $\Delta = \dd \delta = \dd \star \! \dd \star$ and Theorem~\ref{thm:psiec}, iv.).
  Using the primary wavelets $\psi_{j' k'}^{2,\dd}$ as test forms we obtain
  \begin{align}
    \big\langle \! \big\langle \beta \, , \, \psi_{j' k'}^{2,\dd} \big\rangle \! \big\rangle
    &= \Big\langle \! \! \Big\langle \sum_{j=0}^{J} \sum_{k=1}^{\vert \Lambda_j \vert} \omega_{jk} \, \psi_{jk}^{2,\dd} \, , \, \psi_{j' k'}^{2,\dd} \Big\rangle \! \! \Big\rangle
    \\
    \beta_{j'k'} &= \sum_{j=0}^{J} \sum_{k=1}^{\vert \Lambda_j \vert} \omega_s \, D_{jk,j' k'}
  \end{align}
  \end{subequations}
  which, when $j$,$j'$ run only up to some finest level $J < \infty$, is a finite matrix-vector problem amenable to numerical treatment.
  The discrete Laplace operator $D_{jk,j' k'} = \llangle \psi_{jk}^{2,\dd} \, , \, \psi_{j' k'}^{2,\dd} \rrangle$ is thereby by construction invertible, since $\Psiec$ separates the exact forms in $\Omega_{\dd}^2$ from the harmonic forms in $\Omega_h^2$, which form the kernel of $\Delta$.
\end{example}

\begin{example}[Incompressible fluids]
  \label{ex:incompressible_fluid}
  Consider an incompressible fluid with a divergence free velocity vector field $u \in \mathfrak{X}_{\mathrm{div}}(S^2)$.
  Using the musical isomorphisms, $u$ can be associated with a $1$-form field $u^{\flat} \in \Omega_{\delta}(S^2)$ so that the vorticity $\zeta \in \Omega_{\dd}^2(S^2)$ is then given by $\zeta = \dd u^{\flat}$.
  The velocity can, in fact, also be reconstructed from $\zeta$.
  By the Poincar{\'e} lemma, there exists a $\xi \in \Omega_{\dd}^2(S^2)$ such that $u^{\flat} = \delta \xi$ where $\delta$ is the co-differential.
  Taking the exterior derivative of this relation we have $\dd u^{\flat} = \dd \delta \xi = \Delta \xi = \zeta$ so that $u^{\flat} = \delta \Delta^{-1} \zeta$.
  The potential $\xi$ is known as the stream function.

  The above relationships are naturally expressed in $\Psiec$.
  Using the differential form wavelets $\psi_{s}^{1,\delta}$, the velocity $u^{\flat}$ can be written as
  \begin{subequations}
  \begin{align}
    \label{eq:ex:incompressible_fluid:1}
    u^{\flat} = \sum_{j=0}^{J} \sum_{k=1}^{\vert \Lambda_j \vert} u_{jk}^{\flat} \, \psi_{jk}^{1,\delta}
  \end{align}
  where $J = \infty$ for an arbitrary field but it will be finite in numerical calculations.
  The vorticity is hence given by
  \begin{align}
    \zeta = \dd u^{\flat}
    = \sum_{j=0}^{J} \sum_{k=1}^{\vert \Lambda_j \vert} u_{jk}^{\flat} \, \dd \psi_{jk}^{1,\delta}
    = \sum_{j=0}^{J} \sum_{k=1}^{\vert \Lambda_j \vert} u_{jk}^{\flat} \, \psi_{jk}^{2,\dd} .
  \end{align}
  By Theorem~\ref{thm:psiec}, iv.), the stream function has henceforth the representation
  \begin{align}
    \xi = \Delta^{-1} \zeta
    = \sum_{j=0}^{J} \sum_{k=1}^{\vert \Lambda_j \vert} u_{jk}^{\flat} \, \Delta^{-1} \psi_{jk}^{2,\dd}
    = \sum_{j=0}^{J} \sum_{k=1}^{\vert \Lambda_j \vert} u_{jk}^{\flat} \, \tilde{\psi}_{jk}^{2,\dd} .
  \end{align}
  Since $\psi_{s}^{0,\delta} = \star \tilde{\psi}_{s}^{2,\dd}$, by Theorem~\ref{thm:psiec}, ii.), one thus obtains for the reconstruction of the velocity field from the stream function
  \begin{align}
    u^{\flat} = \delta \xi = \star \, \dd \star \xi
    =  \star \, \dd \sum_{j=0}^{J} \sum_{k=1}^{\vert \Lambda_j \vert} u_{jk}^{\flat} \, \psi_{jk}^{0,\delta}
    =  \star \, \sum_{j=0}^{J} \sum_{k=1}^{\vert \Lambda_j \vert} u_{jk}^{\flat} \, \psi_{jk}^{1,\dd}
    =  \sum_{j=0}^{J} \sum_{k=1}^{\vert \Lambda_j \vert} u_{jk}^{\flat} \, \psi_{jk}^{1,\delta}
    \nonumber
  \end{align}
  \end{subequations}
  which is, indeed, the representation of the velocity $1$-form field in Eq.~\ref{eq:ex:incompressible_fluid:1}.
  The above computations can also be traced in Fig.~\ref{fig:psiec_s2}, which can then be read similar to a commutative diagram.
  The present example can be interpreted as $\Psiec$ providing a structure-preserving discretization of incompressible fluids on $S^2$.
\end{example}

The following remarks provide further insight into the properties of $\Psiec$ and relate our construction to existing ones in the literature.

\begin{remark}[Stokes' theorem with wavelet differential forms]
  \label{remark:stokes_thm:wavlets}
We saw in Remark~\ref{remark:stokes_theorem:spectral} that Stokes's theorem can be written using spectral differential form basis functions $y_{lm}^{\fdeg,\nu}$ yielding Eq.~\ref{eq:stokes:spectral:4}.
Following the same steps as there but using the wavelet differential forms $\psi_{jk}^{\fdeg,\nu}$ one obtains
\begin{subequations}
\begin{align}
  \label{eq:stokes:wavelet:1}
   \sum_{jk} \alpha_{jk} \int_{\partial U} \psi_{jk}^{\fdeg,\delta}
  &=  \sum_{jk} \alpha_{jk} \int_{U} \, \psi_{jk}^{\fdeg+1,\dd} .
\end{align}
\end{subequations}
In contrast to spectral differential forms $y_{lm}^{\fdeg,\nu}$, the wavelets $\psi_{jk}^{\fdeg,\delta}$ and $\psi_{jk}^{\fdeg+1,\dd}$ are spatially localized.
Hence, the integrals in Eq.~\ref{eq:stokes:wavelet:1} are non-negligible only for a subset of the wavelets.
On the right hand side, this are the ones whose center $\lambda_{jk}$ is in or sufficiently close to the region $U \subseteq \S2$.
On the left hand side, only the wavelets that are non-negligible over $\partial U$ contribute to the integral, which is a subset of the $\psi_{jk}^{\fdeg+1,\dd}$ on the right hand side.
The integral on the left requires hence, in principle, less computational effort.
Furthermore, because  $\partial U$ has a well localized wavefront set, an approximation is most efficient, i.e. sparsest, when anisotropic, curvelet-like wavelets are used, since then only for those aligned with the boundary yield non-negligible coefficients.

Returning to the case of a spherical cap $\mathcal{C}^{\gamma}$ that was already discussed in Remark~\ref{remark:stokes_theorem:spectral}, the number of wavelets $\smash{\psi_{jk}^{\fdeg+1,\dd}}$ required for a function that is locally over $U$ up to level $J$ is given by $\mathcal{O}( \vert \Lambda_J \vert \, (1 - \cos{\gamma}) )$. In contrast, there are $\mathcal{O}(\sin{\gamma} \sqrt{\vert \Lambda_J \vert})$ non-negligible ones on the boundary.
A further exploration of this, including a rigorous analysis of the convergence rates that would make the above statements precise for arbitrary $\partial U$, is left to future work.
\end{remark}

\begin{remark}[Comparison to WaveTRiSK by Kevlahan and Dubos]
  \label{remark:wavetrisk}
  Kevlahan and Dubos~\cite{Dubos2013,Aechtner2015,Kevlahan2019} proposed a wavelet-based, structure preserving discretization of exterior calculus on $\S2$ that uses second generation, subdivision wavelets.
  As the critical requirement for the multi-resolution structure to be compatible with the discrete exterior calculus~\cite{Thuburn2009,Ringler2010} they identified
  \begin{subequations}
  \begin{align}
    \label{eq:p_dd_commute}
    \mathrm{P}^j \circ \dd^{j+1} = \dd^j \circ \mathrm{P}^j
  \end{align}
  where $\mathrm{P}^j$ is the projection operator from level $j+1$ to level $j$ and $\dd^j$ is the discretized exterior derivative applied on level $j$.
  Eq.~\ref{eq:p_dd_commute} then, for example, ensures mass and vorticity conservation in geophysical fluid dynamics simulations.

  In $\Psiec$, the projection $\mathrm{P}^j$ is realized by dropping the signal representation on level $j+1$.
  Thus for our approach the right hand side of Eq.~\ref{eq:p_dd_commute} is given by
  \begin{align}
     \dd^J \circ \mathrm{P}^J \circ \alpha_{J+1}^{\fdeg,\delta}
     &= \dd^J \circ \mathrm{P}^J \circ \sum_{j=-1}^{J+1} \sum_{k \in \mathcal{K}_j} \alpha_{jk} \, \psi_{jk}^{\fdeg,\delta}
     = \dd^J \circ \sum_{j=-1}^{J} \sum_{k \in \mathcal{K}_j} \alpha_{jk} \, \psi_{jk}^{\fdeg,\delta}
   \end{align}
   where we immediately exploited that only the co-exact part of any form $\alpha$ has a non-trivial exterior derivative.
   Applying Theorem~\ref{thm:psiec} we obtain
   \begin{align}
     \dd^J \circ \mathrm{P}^J \circ \alpha_{J+1}^{\fdeg,\delta}
     = \sum_{j=-1}^{J} \sum_{k \in \mathcal{K}_j} \alpha_{jk} \, \psi_{jk}^{\fdeg+1,\dd} .
  \end{align}
  \end{subequations}
  An analogous calculation shows that the left hand side of Eq.~\ref{eq:p_dd_commute} equals $\mathrm{P}^J  \circ \dd^J \circ \alpha_{J+1}^{\fdeg,\delta}$ so that Eq.~\ref{eq:p_dd_commute} holds.
  Our $\Psiec$ hence also satisfies the requirement put forth by Dubos and Kevlahan.
\end{remark}

%

\begin{remark}
The functional analytic setting of $\Psiec$ are the homogeneous Sobolev spaces $\dot{H}^{-r+1}(\Omega_{\nu}^{\fdeg},S^2)$.
When harmonic forms are treated separately, these are non-degenerate by construction and provide, in our opinion, a natural setting for exterior calculus.
The sequence $H^1 \hookrightarrow L_2 \hookrightarrow H^{-1}$, as it occurs in our representation of $0$-forms and $2$-forms, is a classical example of a Gelfand triple (or rigged Hilbert space), first introduced as a functional analytic setting for the generalized eigenfunctions of the derivative operator.
With this perspective, the primary wavelets $\smash{\psi_{jk}^{\fdeg,\nu}(\omega)}$ and their duals $\smash{\tilde{\psi}_{jk}^{\fdeg,\nu}(\omega)}$ form Gelfand frames, see e.g.~\cite{Feichtinger2009,Trapani2019}, a concept closely related to the Stevenson frames~\cite{Stevenson2003,Balazs2019} that we use.
Stevenson's original work~\cite{Stevenson2003} was, in fact, similar to ours in that he was also interested in the Galerkin-type discretization of operator equations.
 We believe that the Gelfand frame perspective can also be beneficial for $\Psiec$ but we leave a detailed investigation to future work.
\end{remark}

\begin{remark}
In Remark~\ref{remark:stokes_thm:wavlets} we showed that anisotropic differential form wave\-lets could be useful for the numerical realization of Stokes' theorem in $\Psiec$.
Similar to the scalar case, cf. Sec.~\ref{sec:wavelets:discrete}, such form wavelets can be obtained with a straightforward extension of the $\smash{\psi_{jk}^{\fdeg,\nu}(\omega)}$ presented above by introducing mother window coefficients $\smash{\kappa_{lm}^{j,t} = \kappa_{l}^j \, \beta_{m}^{j,t}}$ with a dependence on the azimuthal spherical harmonics parameter $m$.
We hope to address anisotropic differential form wavelets in future work.
\end{remark}

%

\section{$\Psiec$-based Simulation of the Rotating Shallow Water Equations}
\label{sec:shallow}

In the following we will use the local spectral exterior calculus $\Psiec$ that we introduced in the last section to develop a discretization of the rotating shallow water equations.
Numerical results for standard test cases as well as simple forecast experiments will be presented.

\subsection{Exterior Calculus Formulation of the Shallow Water Equations}
\label{sec:shallow:forms}

The shallow water equations in vorticity-divergence form are given by (e.g.~\cite{Williamson1992})
\begin{subequations}
\begin{align}
  \label{eq:shallow:vort_div:1}
  \dot{\zeta} &= -\nabla \cdot (\zeta + f) \, \vec{u}
  \\
  \label{eq:shallow:vort_div:2}
  \dot{\mu} &= \nabla \times (\zeta + f) \, \vec{u} - \Delta \frac{\vert \vec{u} \vert^2}{2} - \Delta g (h + h_e)
  \\
  \label{eq:shallow:vort_div:3}
  \dot{h} &= -\nabla \cdot (h \, \vec{u})
\end{align}
where $h$ is the depth of the fluid and $h_e$ the orography of the earth, $g$ denotes the gravitational constant, and $f$ is the Coriolis parameter that accounts for the rotating frame.
The vorticity $\zeta$ and divergence $\mu$ of the fluid velocity $\vec{u} \in \mathfrak{X}(S^2)$ are
\begin{align}
  \label{eq:shallow:stream_fct_velocity_potential}
  \zeta = \nabla \times \vec{u}
  \quad \quad \quad \quad
  \mu = \nabla \cdot \vec{u} .
\end{align}
The potentials associated with $\zeta$ and $\mu$ are the stream function $\xi$ and the velocity potential $\chi$, respectively.
These are given by
\begin{align}
  \Delta \xi = \zeta
  \quad \quad \quad \quad
  \Delta \chi = \mu .
\end{align}
\end{subequations}
They enable one to write the  velocity vector field as $\vec{u} = \nabla_{\! \perp} \xi + \nabla \chi $ where $\nabla_{\! \perp}$ is the skew-gradient.
To simplify notation, we will write in the following $\eta = \zeta + f$.

By associating the velocity vector field $\vec{u} \in \mathfrak{X}(\S2)$ with the velocity $1$-form $u^{\flat} \in \Omega^1(\S2)$ the shallow water equation can be written using exterior calculus.
It can be shown that Eqs.~\ref{eq:shallow:vort_div:1}-~\ref{eq:shallow:vort_div:3} are then equivalent to
\begin{subequations}
  \label{eq:shallow:ec}
\begin{align}
  \label{eq:shallow:ec:1}
  \dot{\zeta} &= -\dd ( \star \zeta \wedge \star u^{\flat} )
  \\[5pt]
  \label{eq:shallow:ec:2}
  \dot{\mu} &= -\dd (\star \zeta \wedge u^{\flat}) - \Delta \llangle u^{\flat} , u^{\flat} \rrangle - g \, \Delta (h + h_{e})
  \\[5pt]
  \label{eq:shallow:ec:3}
  \dot{h} &= -\dd ( \star h \wedge \star u^{\flat} ) .
\end{align}
\end{subequations}
where the fields are given by
\begin{subequations}
\begin{align}
  u^{\flat} &= u_{\dd}^{\flat} + u_{\delta}^{\flat} \ \ \ \in \Omega_{\dd}^1(S^2) \otimes \Omega_{\delta}^1(S^2)
  \\[3pt]
  \zeta &= \dd u^{\flat} \ \quad \quad \ \in \Omega_{\dd}^2(\S2)
  \\[3pt]
  \mu &= \dd \star u^{\flat} \quad \ \ \in \Omega_{\dd}^2(\S2)
  \\[3pt]
  h & \quad \quad \quad \quad \quad \ \ \in \Omega_{\dd}^2(\S2) .
\end{align}
\end{subequations}

\subsection{Formulation of Rotating Shallow Water Equations in $\Psiec$}
\label{sec:shallow:discretization}

To discretize the rotating shallow water equations in Eq.~\ref{eq:shallow:ec} we write the prognostic fields $\zeta$, $\mu$, and $h$ using the differential form wavelets up to some finest level $J$.
At time $t_n = n \, t_{\Delta}$, with $t_{\Delta}$ being the (fixed) time step, we thus have
\begin{subequations}
\begin{align}
  \zeta(\omega,t_n) = \zeta^n(\omega) &= \sum_{j=0}^{J} \sum_{k=0}^{\vert \Lambda_j \vert} \zeta_{jk}^n \, \psi_{jk}^{2,\dd}(\omega)
  \\
  \mu(\omega,t_n) = \mu^n(\omega) &= \sum_{j=0}^{J} \sum_{k=0}^{\vert \Lambda_j \vert} \mu_{jk}^n \, \psi_{jk}^{2,\dd}(\omega)
  \\
  \mu(\omega,t_n) = h^n(\omega) &= \sum_{j=0}^{J} \sum_{k=0}^{\vert \Lambda_j \vert} h_{jk}^n \, \psi_{jk}^{2,\dd}(\omega) .
\end{align}
\end{subequations}
We will denote the naive vectors of basis function coefficients as $\bar{\zeta}^n$, $\bar{\mu}^n$, and $\bar{h}^n$.

With the above representations, the Hodge dual and the exterior derivative can then be evaluated using the properties of $\Psiec$ in Theorem~\ref{thm:psiec}.
For the wedge product, which currently has no simple expression for the $\psi_{jk}^{\fdeg,\nu}$, we use the transform method~\cite{Orszag1970,Eliasen1970}, i.e. we evaluate the product in the spatial domain and reproject the result onto the wavelets.
By using the wavelet nodes $\lambda_{jk}$ as evaluation points, the reprojection can be implemented efficiently using the exact quadrature discussed in Remark~\ref{remark:wavelets_quadrature}.
This can be interpreted as analysis operator for an $L_j$-bandlimited $\fdeg$-form, which we denote as
\begin{align}
  \label{eq:shallow:anslysis_quadrature}
  \mathcal{A}_{J}^{\fdeg,\nu}(\alpha) = \Bigg\{ \sum_{k'=1}^{\vert \Lambda_J \vert} \alpha(\lambda_{Jk'}) \cdot \tilde{\psi}_{jk}^{\fdeg,\nu}(\lambda_{Jk'}) \Bigg\}_{j = -1, k=1}^{J,\vert \Lambda_j \vert} .
\end{align}
We will ensure that $\mathcal{A}_{J}^{\fdeg,\nu}(\alpha)$ is only required for scalar $0$- and $2$-forms, which simplifies the implementation.
Analogously, the reconstruction operator that determines the spatial representation of a form at the locations $\lambda_{Jk}$ from the basis function coefficients  $\bar{\alpha} = \{ \alpha_{jk}^n \}_{j,k}$ will be denoted as
\begin{align}
  \label{eq:shallow:reconstruction}
  \mathcal{R}_{J}^{\fdeg,\nu}(\bar{\alpha}) = \Bigg\{ \sum_{j=0}^{J} \sum_{k'=0}^{\vert \Lambda_j \vert} \alpha_{jk'}^n \, \psi_{jk'}^{\fdeg,\nu}(\lambda_{J,k}) \Bigg\}_{k=1}^{\vert \Lambda_J \vert} .
\end{align}

%
Using the analysis and reconstruction operators in Eq.~\ref{eq:shallow:anslysis_quadrature} and Eq.~\ref{eq:shallow:reconstruction}, we will detail the computational steps that are required to determine the time evolution of vorticity $\zeta$ in Eq.~\ref{eq:shallow:ec:1}; Eq.~\ref{eq:shallow:ec:2} and Eq.~\ref{eq:shallow:ec:3} follow by analogous considerations.
To avoid having to compute the analysis operator for $1$-forms, we use the Leibniz rule to write Eq.~\ref{eq:shallow:ec:1} as
\begin{align}
  \label{eq:shallow:impl:dot_zeta:1}
  \dot{\zeta} = - \dd \star \zeta \wedge \star u^{\flat} - \star \zeta \wedge \mu .
\end{align}
For the velocity $u^{\flat}(t_n) = u_{\dd}^{\flat}(t_n) + u_{\delta}^{\flat}(t_n)$ at time $t_n = n t_{\Delta}$, we require the stream function $\xi^n = \Delta^{-1} \zeta^n$ and velocity potential $\chi^n = \Delta^{-1} \mu^n$.
Using Theorem~\ref{thm:psiec} we have for the latter
\begin{align*}
  u_{\dd}^{\flat}(t_n) &= \delta \xi^n = \star \dd \star \Delta^{-1} \zeta^n
  = \star \, \dd \star \Delta^{-1} \sum_{jk} \zeta_{jk}^n \, \psi_s^{2,\dd}
  = \star \, \dd \star \sum_{jk} \zeta_{jk}^n \, \tilde{\psi}_s^{2,\dd}
  = \sum_{jk} \zeta_{jk}^n \, \psi_s^{1,\delta}
  \\[3pt]
  u_{\dd}^{\flat}(t_n) &= \dd \! \star \! \chi^n = \dd \star \Delta^{-1} \mu^n
  = \dd \star \Delta^{-1} \sum_{jk} \mu_{jk}^n \, \psi_s^{2,\dd}
  = \dd \star \sum_{jk} \mu_{jk}^n \, \tilde{\psi}_s^{2,\dd}
  = \sum_{jk} \mu_{jk}^n \, \psi_s^{1,\dd}
\end{align*}
i.e. we can obtain $u_{\dd}^{\flat}(t_n)$ and $u_{\delta}^{\flat}(t_n)$ by reconstruction of the $1$-form with the basis function coefficients $\zeta_{jk}^n$ and $\mu_{jk}^n$ of vorticity and divergence, respectively.
For $\dd \star \zeta$ we require $\star \, \zeta$ in the primary basis, which we currently realize using an explicit projection, denoted as $\bar{\star}_{20}$ and given by a matrix.
\begin{subequations}
\label{eq:shallow:impl}
We implement Eq.~\ref{eq:shallow:impl:dot_zeta:1} thus as
\begin{align}
  \label{eq:shallow:impl:dot_zeta}
  \dot{\bar{\zeta}}^n =
  - \mathcal{A}_{J+1}^{2,\dd}\Big(\mathcal{R}_J^{1,\dd}( \bar{\star}_{20} \, \bar{\zeta}^n ) \wedge \big( \mathcal{R}_J^{1,\delta}( \bar{\mu}^n ) + \mathcal{R}_J^{1,\delta}( \bar{\zeta}^n ) \big)
  - \mathcal{R}_J^{0,\delta}\big(\bar{\zeta}^n \big) \wedge \mathcal{R}_J^{2,\dd}(\bar{\mu}^n)\Big)
\end{align}
where the wedge product is evaluated pointwise using its definition in the continuous theory.
As is standard for the transform method, we truncate from level $J+1$ to level $J$ after each time step to remain in the same space over time.

\begin{figure}[t]
  \includegraphics[width=\textwidth]{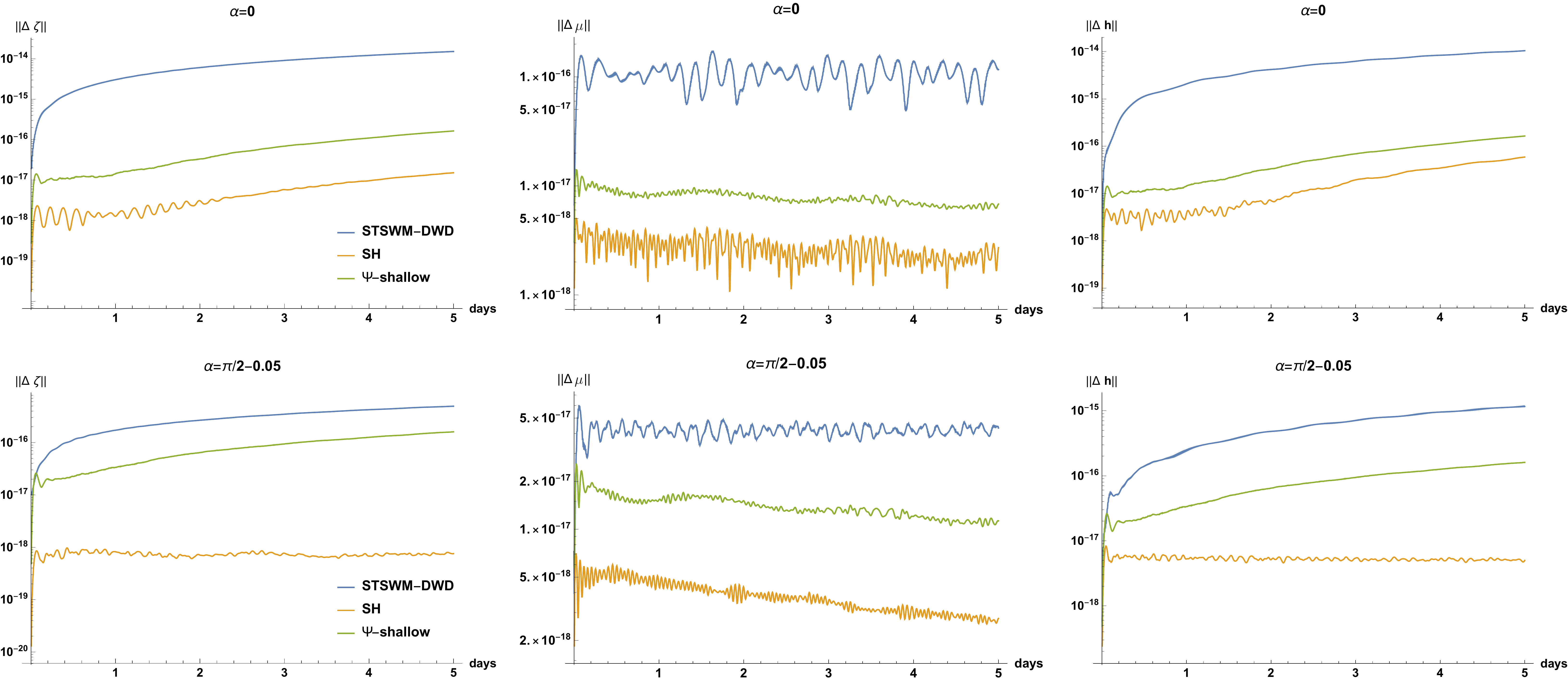}
  \caption{Numerical results the $2^{\textrm{nd}}$ test by Williamson et al.~\cite{Williamson1992} which is a steady state geostrophic flow.
  From left to right we show vorticity $\zeta$, divergence $\mu$, and depth $h$ for two representative values of the parameter $\alpha$ that is the angle between the rotation and the up axis. For all variables the norm of the deviation to the initial value is plotted.}
  \label{fig:williamson2}
\end{figure}

Analogous to vorticity, we obtain for the time evolution of divergence $\mu$ and the fluid depth $h$,
\begin{align}
  \label{eq:shallow:impl:dot_mu}
  \dot{\bar{\mu}}^n &=
    \mathcal{A}_{J+1}^{2,\dd}\Big(\mathcal{R}_J^{1,\dd}( \bar{\star}_{20} \, \bar{\mu}^n ) \wedge \big( \mathcal{R}_J^{1,\delta}( \bar{\mu}^n ) + \mathcal{R}_J^{1,\delta}( \bar{\zeta}^n ) \big)
    + \mathcal{R}_J^{0,\delta}\big(\bar{\mu}^n\big) \wedge \mathcal{R}_J^{2,\dd}(\bar{\zeta}^n)\Big)
    \\[2pt]
    & \quad \quad \quad \quad - \bar{\Delta} \mathcal{A}_{J+1}^{2,\dd}\big( \big( \mathcal{R}_J^{1,\delta}( \bar{\mu}^n ) + \mathcal{R}_J^{1,\delta}( \bar{\zeta}^n ) \big)^2 \big) - g \bar{\Delta} \big( \bar{h}^n + \bar{h}_e \big)
    \nonumber
  \\[5pt]
  \label{eq:shallow:impl:dot_h}
  \dot{\bar{h}}^n &=
  - \mathcal{A}_{J+1}^{2,\dd}\Big(\mathcal{R}_J^{1,\dd}( \star_{20} \, \bar{h}^n ) \wedge \big( \mathcal{R}_J^{1,\delta}( \bar{\mu}^n ) + \mathcal{R}_J^{1,\delta}( \bar{\zeta}^n ) \big)
  - \mathcal{R}_J^{0,\delta}\big(\bar{h}^n \big) \wedge \mathcal{R}_J^{2,\dd}(\bar{\mu}^n )\Big)
\end{align}
\end{subequations}
where $\bar{\Delta}$ is the Galerkin projection of the Laplace-Beltrami operator.
Eqs.~\ref{eq:shallow:impl:dot_zeta}, \ref{eq:shallow:impl:dot_mu}, \ref{eq:shallow:impl:dot_h} provide together our discrete shallow water equations.

For time stepping we use a simple leapfrog scheme with Robertson smoothing, which provided sufficiently accurate solutions in our numerical experiments.

\subsection{Experiments}
\label{sec:shallow:experiments}

In the following we report on experimental results for our $\Psiec$-based discretization of the shallow water equations for the standard tests proposed by Williamson et al.~\cite{Williamson1992} as well as short-time forecast experiments.

\subsubsection{Implementation}

We developed a C++ implementation of Eqs.~\ref{eq:shallow:impl}, which we will refer to as $\Psi$-shallow.
The reported results are for $J=5$.
As reference we use our own spectral implementation, named SH-shallow, based on~\cite{Bourke1972} with libsharp~\cite{Reinecke2013} for the fast spherical harmonics transform.
To have a fair comparison, we chose the bandlimit of the spectral model to match the largest representable frequency of $\Psi$-shallow, i.e. $L_{\mathrm{max}} = 2^{j}$.
We also compared to the implementation by Hack and Jakob~\cite{Hack1992}, in the adaptation developed for the verification of the ICON model~\cite{Korn2018}.
We denote it as DWD-shallow.
All experiments were performed in double precision.

\subsubsection{Standard test cases}

We considered test cases 2, 6, 7 and from~\cite{Williamson1992}, which have been widely used in the literature to assess the correctness of simulations of the shallow water equations.

\begin{figure}[t]
  \includegraphics[width=\textwidth]{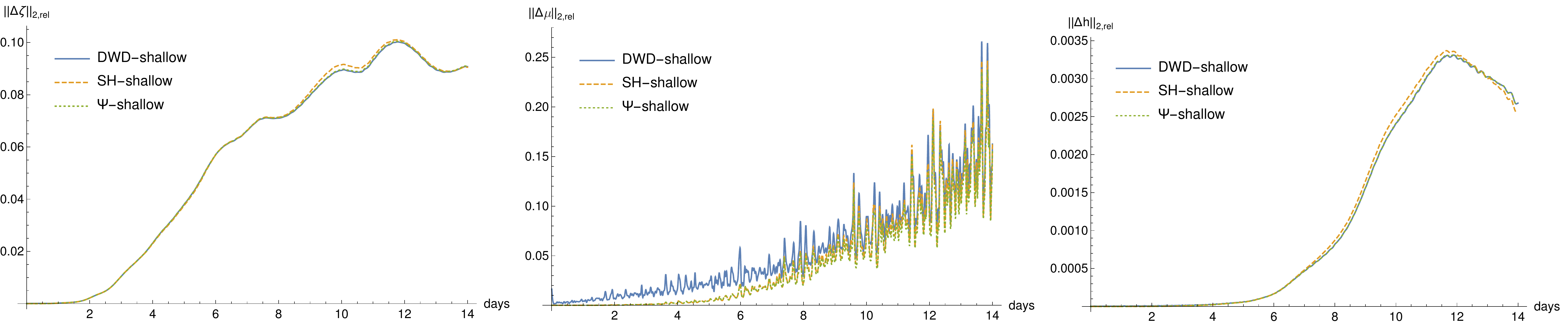}
  \includegraphics[width=\textwidth]{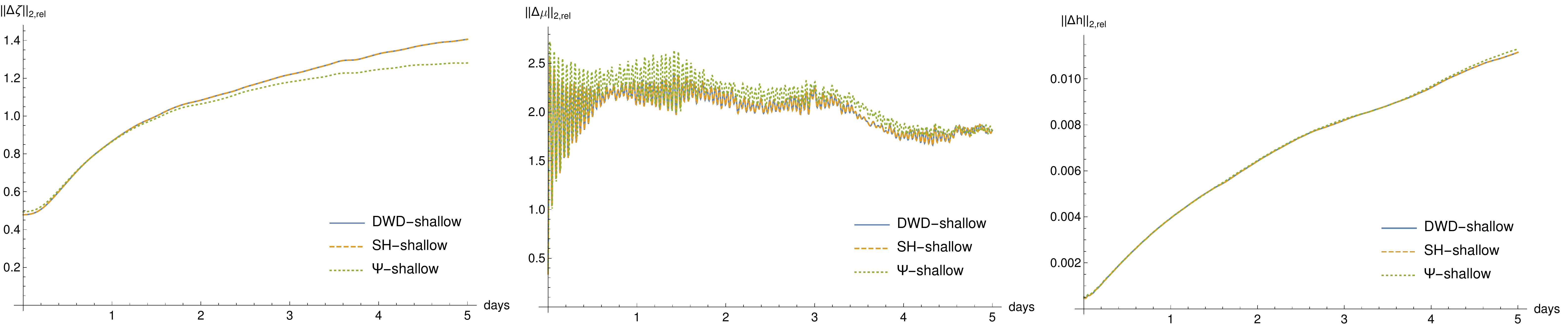}
  \caption{Experimental results the $6^{\textrm{th}}$ (top) and $7^{\mathrm{th}}$ (bottom) standard test by Williamson et al.~\cite{Williamson1992}.
  From left to right we show the deviation of vorticity $\zeta$, divergence $\mu$, and depth $h$ from a reference solution obtained with DWD-shallow with $L_{\mathrm{max}} = 128$.}
  \label{fig:williamson6}
\end{figure}

\paragraph{Test case 2}
This test is a steady state solution with vanishing divergence.
It has a parameter $\alpha$ that is the angle between the rotation axis and the up axis.
Varying $\alpha$ tests the isotropy of the model, e.g. if flows over the pole can be represented as accurately as those along the equator.
Fig.~\ref{fig:williamson2} shows the norm of the deviations of vorticity, divergence and geopotential from the initial value for a 10 day simulation.
All three implementations preserve the initial values to high  accuracy, i.e. they provide good simulations of the expected steady state.
The slightly larger error for $\Psi$-shallow compared to SH-shallow results from the fact that the tight frame property is enforced numerically and a residual slightly larger than machine precision remains at the end of the optimization.
In contrast, libsharp, used in SH-shallow, performs highly accurate spherical harmonics transforms with an error on the order of machine precision.

\paragraph{Test case 6}
Fig.~\ref{fig:williamson6} shows results for test case 6, which is a Rossby-Hurrwitz wave.
No analytic solution is available in this case so we used DWD-shallow with $L_{\mathrm{max}} = 128$ as such.
The results demonstrate that the $\Psiec$-discretization provides accuracy comparable with those obtained by our spectral implementation for all three prognostic variables $\zeta$, $\mu$, and $h$.

\paragraph{Test case 7}
The test considers physical initial conditions for January 1979.
We again use DWD-shallow with $L_{\mathrm{max}} = 128$ as reference.
Although a slight deviation of the solution of $\Psi$-shallow can be seen over time, it remains sufficiently close to provide accurate predictions.

\paragraph{Energy and enstrophy}
In Fig.~\ref{fig:williamson6_en} we show the change in energy $E = \llangle u^{\flat} , u^{\flat} \rrangle$ and enstrophy $\mathcal{E} = \llangle \zeta , \zeta \rrangle$ for test case 6, i.e. the Rossby-Hurrwitz wave, for $\Psi$-shallow.
The result demonstrate excellent conservation properties for our implementation based on $\Psiec$, as one would expect with its respect for exterior calculus.
A theoretical analysis of the conservation properties of $\Psi$-shallow will be presented in a forthcoming publication.


\subsubsection{Forecast experiments}

To obtain some insight on the performance of our discretization under more realistic conditions we performed forecast experiments using reanalysis data (ERA-Interim~\cite{Dee2011}).
We used each time slice available in the data set as initial condition and ran the simulation for $6$ hours.
We then compared the forecast to the data for the time point in the reanalysis.
As naive base line we used the persistent forecast where the data is kept constant over the $6$ hour period.

In Fig.~\ref{fig:reanalysis} we report the difference between forecast and reanalysis data for the year 1979; analogous ones hold for other years.
The plots show that our simulations provide substantial improvements over a naive forecast especially for vorticity.
For divergence there is a smaller improvement and  $\Psi$-shallow is less accurate than SH-shallow.



\section{Future Work}
\label{sec:future_work}

The presented results provide many avenues for future work.
Our long term objective is the development of a data-assisted dynamical core for the prediction of climate statistics.
For this, we want to extend the discretization of the shallow water equations developed in Sec.~\ref{sec:shallow:discretization} to one for the hydrostatic primitive equations and couple it to neural networks that ensures the correct prediction of local statistics.
Preliminary experiments indicate that the differential form wavelets provide a useful representation of the data for the neural networks, which also ensures that these respect the basic physical principles encoded in the Hodge-Helmholtz decomposition.

\begin{figure}[t]
  \includegraphics[width=\textwidth]{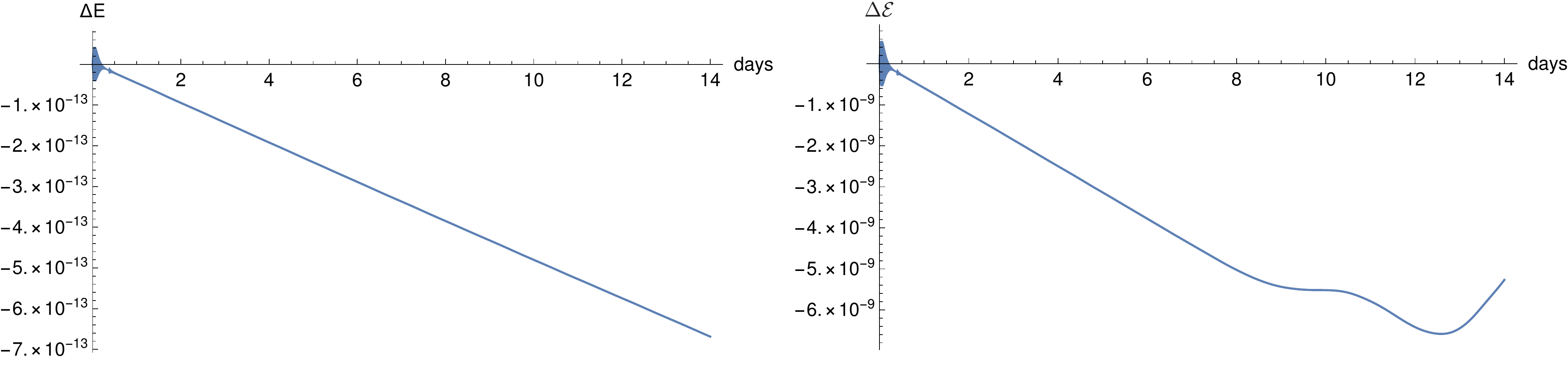}
  \caption{Change in energy $E$ and enstrophy $\mathcal{E}$ for the the $6^{\textrm{th}}$ test by Williamson et al.~\cite{Williamson1992}.}
  \label{fig:williamson6_en}
\end{figure}

Our local spectral exterior calculus for the sphere $\Psiec$ can be developed further in different directions.
Currently, we only consider differential forms, analogous to other existing discretizations of exterior calculus, e.g.~\cite{Desbrun2006,Arnold2018}.
However, since our wavelets are forms in the sense of the continuous theory they naturally pair with vector fields.
This suggests to extend our construction by frame representations for vector fields.
Then, for example, the Lie derivative could be evaluated directly.
This would considerably simplify many equations, for instance Eq.~\ref{eq:shallow:ec} could be written and implemented much more directly.
In our $\Psiec$-based implementation of the shallow water equations, some terms, such as $\dd \star \zeta$, require an explicit projection, which can become a computational bottleneck.
It should hence be investigated how such projections can be avoided or if efficient mass lumping-like implementations are possible, similar to what has been accomplished in Discrete Exterior Calculus~\cite{Desbrun2006}.

In future work, we would also like to investigate the approximation properties of our differential form wavelets.
In the scalar case, similar questions have been investigated for compactly supported multi-scale RBFs
~\cite{LeGia2010,LeGia2012} and the Sobolev space setting has also been considered by Freeden and co-workers~\cite[Ch. 5]{Freeden1998}.
To the best of our knowledge, the case of differential forms has not been investigated.
Of interest is in this context also the utility of anisotropic differential form wavelets, which, as we already discussed in Sec.~\ref{sec:forms}, can be obtained with a straightforward extension of the construction in the present work.
We conjecture that, analogous to the scalar case~\cite{Candes2004,Candes2005a,Candes2005b}, these are required to attain (quasi-)optimal approximation rates for directional fields, e.g. flows along boundaries or global circulation patterns.
The results on the approximation properties are a prerequisite for the development of adaptive numerical schemes that exploit varying regularity, e.g.~\cite{Stevenson2003}.
Such schemes are another long term objective we would like to pursue.

\begin{figure}[t]
  \includegraphics[width=\textwidth]{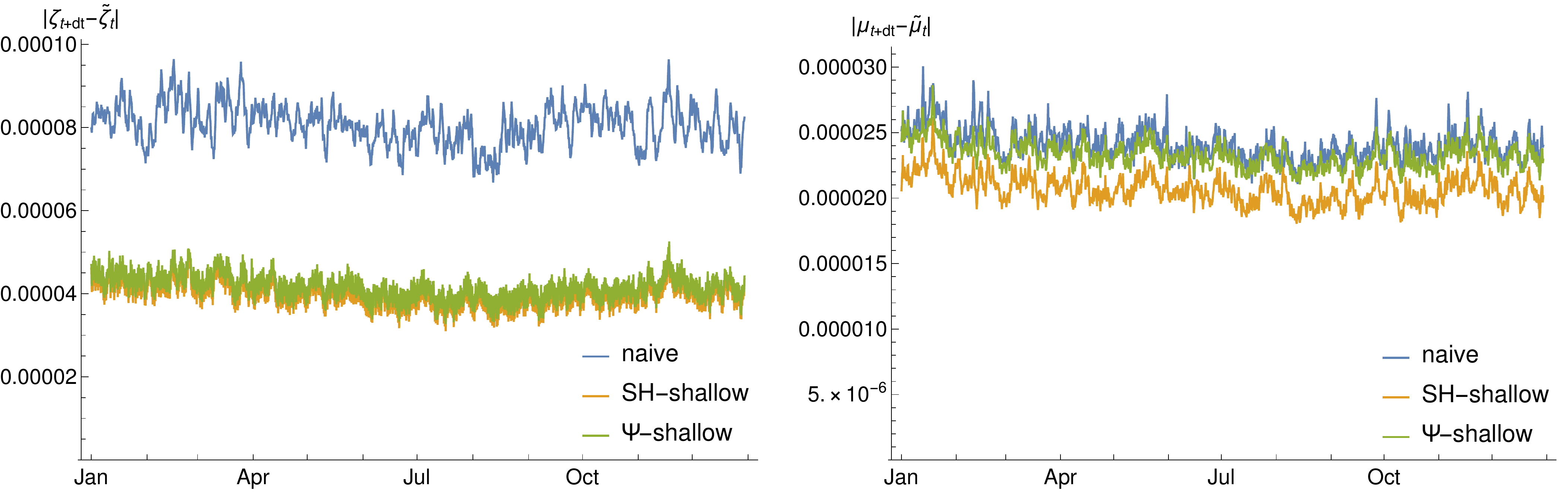}
  \caption{Results for $6$ hours forecast experiments with reanalysis data (ERA-Interim~\cite{Dee2011}) as initial conditions. In the naive forecast that we used as reference the data is kept constant over the $6$ hours period. Reported is the difference between forecast and reanalysis data for all time slices in 1979.}
  \label{fig:reanalysis}
\end{figure}

The construction of structure preserving numerical integrators based on $\Psiec$ is another interesting direction for future work.
The numerical results presented in Sec.~\ref{sec:shallow} indicate that $\Psiec$ yields energy and enstrophy conservation naturally when the vorticity-divergence formulation of the shallow water equations is used.
A theoretical analysis will be presented in a forthcoming publication.

\section{Conclusions}
\label{sec:conclusions}

In this paper we introduced $\Psiec$, a wavelet-based discretization of exterior calculus for the two-sphere $\S2$.
It is based on differential form wavelets $\smash{\psi_{jk}^{\fdeg,\nu}(\omega)}$ that provide (Stevenson) frames for homogeneous Sobolev spaces $\smash{\dot{H}^{-r+1}(\Omega_{\nu}^{\fdeg},S^2)}$.
These were derived from needlet-like, scalar wavelets that we obtained using scalable reproducing kernel frames.
In contrast to other discretizations of exterior calculus, $\Psiec$ systematically distinguishes between exact, co-exact and harmonic forms, which provides precise control domain, image and kernel of the exterior derivative.

Using $\Psiec$, we developed a discretization of the rotating shallow water equations.
Our numerical experiments for standard test cases and forecast experiments demonstrate that it provides accuracy comparable to classical spectral methods and preserves energy and enstrophy.
In future work, we want to extend this discretization to the hydrostatic primitive equations.

\section*{Acknowledgments}
CL would like to thank Mathieu Desbrun for helpful discussion on the relationship between Discrete Exterior Calculus and the present work.
Funding by AIR Worldwide is gratefully acknowledged.

\appendix
\section{Scalable Reproducing Kernel Frames for $H_{\leq L_j}(\S2)$}
\label{sec:scalable_rk_frames:optimization}


In the following, we will detail the numerical construction of the scalable reproducing kernel frames $(\Lambda_j,w_j)$ spanning the spaces $\mathcal{H}_{\leq L_j}(\S2)$ that are used for the discretization of the scalar wavelets in Sec.~\ref{sec:wavelets} and differential form wavelets in Sec.~\ref{sec:forms}.
Since the construction is not specific to any $L_j$ we will omit the level index $j$ in the following.

Since we currently do not have a closed form method for the construction of the spherical scalable reproducing kernel frames, the locations $\lambda_k \in \S2$ and weights $w_k \in \R^+$ forming $(\Lambda, w)$ spanning $\mathcal{H}_{\leq L}(\S2)$ are obtained using nonlinear numerical optimization. 
To characterize the quality of $(\Lambda,w)$, we directly use the deviation from the scalable frame property, i.e.
\begin{align}
  \label{eq:scalable_rk:energy:basic}
  E(\Lambda,w) = \big\Vert \tilde{S} \big\Vert_{\mathrm{F}} = \big\Vert K^H W K - \mathrm{Id} \big\Vert_{\mathrm{F}} ,
\end{align}
where $K \in \R^{\vert \Lambda \vert \times N}$ with $N = (L+1)^2$ is the kernel matrix whose entries are $K_{k,l^2+l+m} = y_{lm}(\lambda_k)$ and $W \in \mathbb{R}^{\vert \Lambda \vert \times \vert \Lambda \vert}$ is the diagonal matrix formed by the weights $w_k$.
We use the Frobenius norm $\Vert \cdot \Vert_{F}$ for $E(\Lambda,w)$ in Eq.~\ref{eq:scalable_rk:energy:basic} since it facilitates the computation of the gradient $\smash{\nabla E = ( \nabla_{w_k} E(\Lambda,w) , \nabla_{\theta_k} E(\Lambda,w) , \nabla_{\phi_k} E(\Lambda,w))_{k=1}^{\vert \Lambda \vert}}$.
Its components, in a form suitable for numerical computations, are
\begin{subequations}
  \label{eq:scalable_rk:gradient}
\begin{align}
  \nabla_{w_k} E(\Lambda,w) &= \frac{1}{E} K_k^H \, (\tilde{S}^* \, K)_k
  \\[4pt]
  \nabla_{\theta_k} E(\Lambda,w) &= \frac{2}{E} (K_{\theta})_k^H \, (\tilde{S}^* \, K)_k
  \\[4pt]
  \nabla_{\phi_k} E(\Lambda,w) &= \frac{2}{E} (K_{\phi})_k^H \, (\tilde{S}^* \, K)_k .
\end{align}  
\end{subequations}
The matrices $K_{\theta}$ and $K_{\phi}$ are formed by the derivatives of the spherical harmonics, i.e. $(K_{\theta})_{k,l^2+l+m} =  \partial y_{lm}(\lambda_k) / \partial \theta$ and $(K_{\phi})_{k,l^2+l+m} = \partial y_{lm}(\lambda_k) / \partial \phi$, and $K_k^H$ refers to the $k^{\mathrm{th}}$ row of the matrix $K^H$.

\begin{table}[t]
\centering 
  \begin{tabular}{c|c|c|c|c|c}
  $j$ & $\vert \Lambda_j \vert $ & $\mathrm{min}(E)$ & $\mathrm{min}(E) / \vert \Lambda_j \vert $ & time (sec) & iterations (phase 1)
  \\[2pt]
  \hline
  & & & & 
  \\[-10pt]
  2 & 32 & 9.19787e-16 & 2.87433e-17 & 2 & 285 (149)
  \\[2pt]
  3 & 128 & 3.71656e-15 & 2.90356e-17 & 10 & 1439 (1062)
  \\[2pt]
  4 & 512 & 2.94946e-14 & 5.76066e-17 & 61 & 1739 (875)
  \\[2pt]
  5 & 2048 & 7.24020e-13 & 3.53525e-16 & 5533 & 6225 (2656)
  \\[2pt]
  6 & 8192 & 1.22619e-11 & 1.49681e-15 & 807 $\times 10^3$ & 22482 (6292) 
 \end{tabular}
 \caption{Results of the nonlinear optimization to obtain scalable reproducing kernel frames $( \Lambda_j , w_j )$ for different levels $j$. Timings are for a shared memory implementation with 32 threads. The value in brackets for the iteration number refers to those required for the first optimization phase where only the locations are optimized.}
 \label{tab:opt}
\end{table}

With the above energy and gradient, numerical optimization of $(\Lambda, w)$ can be realized.
We perform the it in two phases to facilitate well distributedness of the locations $\lambda_k$. 
In the first phase, only the $\lambda_k$ are optimized and the weights $w_k$ are fixed at the ideal value $w_k = 4\pi / \vert \Lambda \vert$.
This yields well distributed points since the weights $w_k$ are a means to compensate for a lack of uniformness. 
In the second phase, both the locations and the weights are variable. 
We observed that the locations change only by small amounts in this phase.
The optimization is implemented in custom C++ code with the minimization performed using the conjugate gradient method available in the ALG library~\cite{Bochkanov2020}.
To reduce computations times, the code has been parallelized for shared memory systems with the construction of $K$ and also the products required for the gradient evaluated by multiple threads simultaneously. 

Table.~\ref{tab:opt} shows final energies and the optimization times (for 32 threads on Intel(R) Xeon(R) Gold 5122 CPU @ 3.60GHz with 16 Core CPU). 
Experiments with extended precision indicate that the optimization yields true minimizers. 
Since our scalable reproducing kernel frames use nested locations, i.e. $\Lambda_j \subset \Lambda_{j+1}$, the optimizations for different $j$ have to be performed in order. 
The additional points for the next level are thereby always obtained from a quasi random sequence on $[0,1]^2$ mapped to the sphere with an area preserving mapping.

The results in Table~\ref{tab:opt} indicate that with a shared memory implementation levels $j$ with $j > 6$ will require an excessive amount of computation time.
We hence implemented an MPI-based task parallel version of the optimization that can run on large cluster computers. 
Its details and the results we obtained with it will be presented in a forthcoming publication.

\bibliographystyle{siamplain}
\bibliography{climate,saam,psiec}

\end{document}